\documentclass[a4paper,10pt]{scrartcl}
\usepackage[fleqn]{amsmath}
\usepackage{amssymb}
\usepackage{amsthm}
\usepackage{stmaryrd}
\usepackage{bbm}
\usepackage{hyperref}
\usepackage{mathtools}
\usepackage[numbers]{natbib}

\mathtoolsset{showonlyrefs}

\newcommand{\bE}{\ensuremath{\mathbb{E}}}

\newcommand{\bH}{\ensuremath{\mathbb{H}}}

\newcommand{\bN}{\ensuremath{\mathbb{N}}}

\newcommand{\bP}{\ensuremath{\mathbb{P}}}
\newcommand{\bQ}{\ensuremath{\mathbb{Q}}}
\newcommand{\bR}{\ensuremath{\mathbb{R}}}

\newcommand{\bZ}{\ensuremath{\mathbb{Z}}}

\newcommand{\ind}{\ensuremath{\mathbbm{1}}}

\newcommand{\cA}{\ensuremath{\mathcal{A}}}

\newcommand{\cC}{\ensuremath{\mathcal{C}}}

\newcommand{\cF}{\ensuremath{\mathcal{F}}}

\newcommand{\cL}{\ensuremath{\mathcal{L}}}

\newcommand{\abs}[1]{\left\vert \, #1 \, \right\vert}
\newcommand{\norm}[1]{\left\Vert \, #1 \, \right\Vert}

\newcommand{\ddx}[1][1]{\ifnum#1=1 \frac{d}{dx} \else \frac{d^{#1}}{dx^{#1}} \fi}
\newcommand{\ddy}[1][1]{\ifnum#1=1 \frac{d}{dy} \else \frac{d^{#1}}{dy^{#1}} \fi}
\newcommand{\ddt}[1][1]{\ifnum#1=1 \frac{d}{dt} \else \frac{d^{#1}}{dt^{#1}} \fi}

\usepackage{todonotes} 

\theoremstyle{plain}
\newtheorem{thm}{Theorem}[section]  

\newtheorem{cor}[thm]{Corollary}
\newtheorem{lem}[thm]{Lemma}
\newtheorem{defn}{Definition}[section]

\theoremstyle{definition}

\theoremstyle{remark}
\newtheorem{rem}{Remark}[section]

\numberwithin{equation}{section}

\DeclareMathOperator{\integers}{\mathbb{Z}}
\DeclareMathOperator{\reals}{\mathbb{R}}

\DeclareMathOperator{\nat}{\mathbb{N}}

\DeclareMathOperator{\Prob}{\mathbb{P}}

\newcommand{\set}[1]{\{ #1 \}}
\newcommand{\Lat}{\mathbb{Z}^{d+1}}
\newcommand{\Latd}{\mathbb{Z}^d}

\newcommand{\hP}{\widehat{\bP}}
\newcommand{\tP}{\widetilde{\bP}}
\newcommand{\hE}{\widehat{\bE}}
\newcommand{\tE}{{\widetilde{E}}}
\newcommand{\td}{{\widetilde{d}}}
\newcommand{\tmu}{{\widetilde{\mu}}}
\newcommand{\tOmega}{{\widetilde{\Omega}}}
\DeclareMathOperator{\sign}{sign}

\title{Absolute continuity and weak uniform mixing of random walk in dynamic
random environment}

\author{ 
Stein Andreas Bethuelsen
 \footnote{Technische Universit\"at M\"unchen, Fakult\"at f\"ur Mathematik, Boltzmannstr.\ 3, 85748 Garching, Germany. 
 Email: stein.bethuelsen@tum.de}  
 \quad
 Florian V\"ollering   
 \footnote{Department of Mathematical Sciences, University of Bath, Claverton Down, Bath, BA2 7AY, United Kingdom}
}

\begin{document}

\maketitle

\begin{abstract}

We prove results for random walks in dynamic random environments which do not
require the strong uniform mixing assumptions present in the literature. We
focus on the ``environment seen from the walker"-process and in particular its
invariant law. Under general conditions it exists and is mutually absolutely
continuous to the environment law. With stronger assumptions we obtain for
example uniform control on the density or a quenched CLT. The general
conditions are made more explicit by looking at hidden Markov models or Markov
chains as environment and by providing simple examples.
\end{abstract}

\bigskip

\bigskip

\section{Introduction and main results}

\subsection{Background and motivation}

We study the asymptotic behaviour of a class of random walks $(X_t)$ on
$\mathbb{Z}^{d}$ whose transition probabilities depend on another process, the
random
environment. Such models play an important role in the understanding of
disordered systems and serve as natural generalisations of the
classical
simple random walk model for describing transport processes in
inhomogeneous
media.

These types of random walks, which are called random walks in random
environment, can be split into two broad areas, static and dynamic
environments. In static environments the environment is created
initially and
then stays fixed in time. In dynamic environments the environment
instead
evolves over time. Note that a dynamic environment in $\bZ^d$ can
always be
reinterpreted as a static environment in $\bZ^{d+1}$ by turning time
into
an additional space dimension.

A major interest in dynamic environments are their often complicated
space-time dependency structure. Typically, in order to show that the random walk is diffusive, one looks for some way to
guarantee that the environment is ``forgetful'' and random walk
increments are
sufficiently independent on large time scales. 

One approach to this are various types of mixing assumptions on the
environment.
By now, general results are known for Markovian environments which are
uniformly mixing with respect to the starting configuration (see 
\citet*{AvenaHollanderRedigRWDRELLN2011} and
\citet{RedigVolleringRWDRE2013}).
For this, the rate at which the dynamic environment converges towards
its
equilibrium state plays an important role.

On the other hand, models where the dynamical environment has
non-uniform
mixing properties  serve as a major challenge  and are still not well
understood. Opposite to the diffusive behaviour known for uniform
mixing
environments, it has been conjectured that $(X_t)$ may be sub- or
super-diffusive for certain non-uniform mixing environments, see 
\citet{AvenaThomannRWDREsimulations2012}. Though some particular
examples
yielding diffusive behaviour have recently been studied by rigorous methods, e.g.\
\citet*{DeuschelGuoRamirezRWBDRE2015}, 
\citet*{HilarioHollanderSidoraviciusSantosTeixeiraRWRW2014},
\citet{HuveneersSimenhausRWSEP2014}
and \citet{MountfordVaresDCP2013}, these results are model specific
and/or
 perturbative in nature.
 No general theory has so far been developed.

In this article we provide a new approach for determining limiting
properties
of random walks in dynamic random environment, in particular about the
invariant law of the
``environment as seen from the walker''-process. Under general mixing assumptions, we prove  the existence of an invariant measure mutually absolutely
continuous with
respect to the random environment (Theorems \ref{def assump1} and
\ref{def
assump2}). Our mixing assumptions are considerably weaker than the uniform mixing conditions present in the literature (e.g.\ cone mixing) and do not require the environment to be Markovian.

An important feature of our approach is that it can also be applied to
dynamics with non-uniform mixing properties. Examples include an
environment
given by Ornstein-Uhlenbeck processes and the supercritical contact process.

Knowledge about the invariant measures for the ``environment as seen from the walker''-process yield limit laws for the random walk itself. One immediate application of  our approach is a strong law of large numbers for the random walk. Further applications include a quenched CLT based on
\citet*{DolgopyatKellerLIveraniRWME2008}, Theorem 1,
considerably relaxing its requirements.

Our key observation is an expansion of the ``environment as seen from
the
walker''-process (Theorem \ref{thm expansion}). This expansion enables
us to
separate the contribution of the random environment  to the law of the
``environment as seen from the walker''-process from that of the
transition
probabilities of the random walk.

We also show stability under perturbations of the environment or of the jump
kernel
of the random walk. Under a strong uniform mixing assumption, we obtain
uniform
control on the Radon-Nikodym derivative of the law of the ``environment
as
seen from the walker''-process with respect to the environment,
irrespective of
the choice of the jump kernel of the random walker.

\subsubsection*{Outline}
In the next two subsections we give a precise definition of our model and present our main results, Theorem \ref{def assump1} and Theorem \ref{def assump2}. Section \ref{sec examples} is devoted to examples and applications thereof. In Section \ref{sec expansion} we derive the aforementioned expansion, and present results on stability and control on the Radon-Nikodym derivative. Proofs are postponed until Section \ref{sec proofs Florian}.

\subsection{The model}
In this subsection we give a formal definition of our model. In short, $(X_t)$ is a random walk in a translation invariant random field with a deterministic drift in a fixed coordinate direction.

\subsubsection*{The environment}\label{sec model env}

Let $d\in \nat$ and let $\Omega := E^{\mathbb{Z}^{d+1}}$ where $E$ is assumed to be a finite set. 
We assign to the space $\Omega$ the standard product $\sigma$-algebra $\mathcal{F}$ generated by the cylinder events.  For $\Lambda \subset \mathbb{Z}^{d+1}$, we denote by $\mathcal{F}_{\Lambda}$ the sub-$\sigma$-algebra generated by the cylinders of $\Lambda$. For the forward half-space $\mathbb{H}:= \mathbb{Z}^{d} \times \integers_{\geq 0}$ we write $\mathcal{F}_{\geq0}$ for $\mathcal{F}_{\bH}$.

By $\mathcal{M}_1(\Omega)$ we denote the set of probability measures on $(\Omega, \mathcal{F})$.
We call $\eta \in \Omega$ the \emph{environment} and denote by $\bP \in  \mathcal{M}_1(\Omega)$ its law.
A particular class of environments contained in our setup are path measures of a stochastic process $(\eta_t)$ whose state space is $\Omega_0 :=E^{\mathbb{Z}^{d}}$.
To emphasise this, for $\eta \in \Omega$ and $(x,t) \in \integers^d \times \integers$, we often write  $\eta_t(x)$ for the value of $\eta$ at $(x,t)$.

 We assume throughout that $\bP$ is measure preserving with respect to translations, that is, for any $x\in \mathbb{Z}^{d}$, $t\in \integers$,
\begin{align}\label{eq translation invariant}
\bP(\cdot) = \bP(\theta_{x,t} \cdot ),
\end{align}
 where $\theta_{x,t}$ denotes the shift operator $\theta_{x,t} \eta_s(y) = \eta_{s+t}(y+x)$. Furthermore, we assume that $\bP$ is \emph{ergodic in the time direction}, that is, all events $B \in \mathcal{F}$ for which $\theta_{o,1} B :=  \{\omega \in \Omega \colon \theta_{o,-1} \omega \in B\} = B$ are assumed to satisfy $\bP(B) \in \{0,1\}$. Here,  $o\in \mathbb{Z}^{d}$ denotes the origin.
 
 \begin{rem}
We present in Subsection \ref{sec HMM}-\ref{sec OU} an approach where $E$ is allowed to be a general Polish space, by considering the environment as a hidden Markov model.
\end{rem}

\subsubsection*{The random walk}\label{sec RW}
The \emph{random walk} $(X_t)$ is a process on $\integers^d$. We assume w.l.o.g.\ that $X_0=o$.  The transition probabilities of $(X_t)$ is assumed to depend on the state of the environment as seen from the random walk. That is, given $\eta \in \Omega$, then the evolution of $(X_t)$ is given by
\begin{align}
&P_{\eta}(X_0 = o) =1
\\ &P_{\eta}(X_{t+1} = y+z \mid X_t = y) = \alpha( \theta_{y,t} \eta, z),
\end{align}
where $\alpha \colon \Omega \times \integers^d \rightarrow [0,1]$ satisfies $\sum_{z \in \mathbb{Z}^{d}} \alpha (\eta,z) =1$ for all $\eta \in \Omega$. The law of the random walk, $P_{\eta} \in \mathcal{M}_1((\Latd )^{\integers_{\geq0}})$, where we have conditioned on the entire environment, is called the \emph{quenched} law. We denote its $\sigma$-algebra by $\mathcal{G}$. 
Further, for $\mathbb{P} \in \mathcal{M}_1(\Omega)$, we denote by $P_{\bP}  \in \mathcal{M}_1 \left( \Omega \times (\Latd )^{\integers_{\geq0}}\right)$ the joint law of $(\eta,X)$, that is,
\begin{equation}
P_{\bP}(B \times A) = \int_{B}P_{\eta}(A) d\bP(\eta), \quad B \in \mathcal{F}, A \in \mathcal{G}.
\end{equation}
 The marginal law of $P_{\bP}$ on $(\Latd)^{\integers_{\geq0}}$ is the \emph{annealed} (or averaged) law of $(X_t)$.

We assume that the transition probabilities of $(X_t)$ only depend on the environment within a finite region around its location. That is, there exist $R \in \nat$ such that
for all $z \in \mathbb{Z}^{d}$
\begin{align}\label{eq assump alpha}
 \alpha(\eta,z)- \alpha(\sigma,z) = 0  \text{ whenever } \sigma \equiv \eta \text{ on } [-R,R]^{d} \times \{0\}.  
 \end{align}
Further, define
\begin{equation}\label{eq walker range}
\mathcal{R} := \set{y \in \mathbb{Z}^{d} \colon \sup_{\eta \in \Omega} \alpha(\eta,y)>0 }
\end{equation} as the jump range of the random walker, which we assume to be finite and to contain $o$. By possibly enlarging $R$ we can  guarantee that 
\begin{align}\label{eq walker R}
\sup_{y \in \mathbb{Z}^{d}} \set{ \norm{y}_1 \colon y \in \mathcal{R}} \leq R.
\end{align} 
  Lastly, we say that $(X_t)$ is \emph{elliptic in the time direction} if 
\begin{align}\label{eq weakly elliptic Florian}
\alpha(\eta,o) >0, \quad \forall \: \eta \in \Omega.
\end{align}
  If, after replacing $o$ with $y \in \mathcal{R}$, \eqref{eq weakly elliptic Florian} holds for all $y \in \mathcal{R}$, then we say that  $(X_t)$ is  \emph{elliptic}. 
 
\subsubsection*{The environment process}\label{sec EP def}

``The environment as seen from the walker''-process is of importance for understanding the asymptotic behaviour of the random walk itself, but it is also of independent interest. This process, which is given by
\begin{equation} (\eta_t^{EP}):=(\theta_{X_t,t}\eta), \quad t \in \integers_{\geq0}, \end{equation}
is called the \emph{environment process}. 
Note that $(\eta_t^{EP})$ is a Markov process on $\Omega$ under $P_{\eta}$, $\eta \in \Omega$, with initial distribution $\bP$.

 \subsection{Main results}\label{sec results Florian}
In this subsection we present our main results about the asymptotic behaviour of $(X_t)$ and $(\eta_t^{EP})$. However, before stating our first theorem we need to introduce some more notation.

Recall \eqref{eq walker range} and let
 \begin{equation}\label{eq Gamma k}
 \Gamma_k:= \left\{(\gamma_{-k},\gamma_{-k+1},...,\gamma_0) \colon \gamma_i \in \bZ^d,  \gamma_i-\gamma_{i-1}\in \mathcal{R} , -k\leq i<0, \gamma_0=o \right\}
 \end{equation}
 be the set of all possible backwards trajectories from $(o,0)$ of length $k$. For $\gamma \in \Gamma_k$ and $\sigma \in \Omega$, denote by
\begin{equation}\label{eq Akm}
 A_{-k}^{-m}(\gamma, \sigma) := \bigcap_{i=-k}^{-m} \left\{ \theta_{\gamma_i,-i}\eta \equiv \sigma_i \text{ on } [-R,R]^{d}\times \{0\} \right\}, \quad 1 \leq m \leq k,
\end{equation}
the event that an element $\eta \in \Omega$ equals $\sigma $ in the $R$-neighbourhood along the path $(\gamma_k, \dots, \gamma_{-m})$.
$A_{-k}^{-1}(\gamma,\sigma)$ is the event that the path of the environment observed by the random walk equals $\sigma$ if the random walk moves along the path $\gamma$.
 Given $\gamma \in \Gamma_k$, denote by 
\begin{align}
\mathcal{A}_{-k}^{-m}(\gamma) := \left\{ A_{-k}^{-m}(\gamma,\sigma) \colon \sigma \in \Omega \text{ and } \bP(A_{-k}^{-m}(\gamma,\sigma))>0 \right\}
\end{align} 
the set of all possible observations along the path $(\gamma_k, \dots, \gamma_{-m})$.
We write
$\mathcal{A}_{-\infty}^{-m}$ for the set of  events $\bigcup_{k \geq m, \gamma \in {\Gamma_k}} \mathcal{A}^{-m}_{-k}(\gamma)$. If $m=1$ we simply write $\mathcal{A}_{-\infty}$. 

 Further, denote by $\mathcal{C} := \{ (x,t) \in \bH \colon \norm{x}_1 \leq (R+1)t \}$ the forward cone with centre at $(o,0)$ and slope proportional to $R+1$. For $j \in \integers$, denote by $\mathcal{C}(j) := \mathcal{C} \cap \theta_{o,j} \bH$ 
and let $\mathcal{F}_{\infty}^{\infty} := \bigcap_{j \in \nat} \mathcal{F}_{\mathcal{C}(j)}$ be the tail-$\sigma$-algebra with respect to $\mathcal{F}_{\mathcal{C}}$.

\begin{thm}[Existence of an ergodic measure for the environment process]\label{def assump1}
Assume that $\bP \in \mathcal{M}_1(\Omega)$ satisfies 
\begin{align}\label{eq cone mixing Florian}
\lim_{l \rightarrow \infty}
 \sup_{B \in \mathcal{F}_{\mathcal{C}(l)}}
\sup_{A \in \mathcal{A}_{-\infty}} \left|\bP(B \mid A) - \bP(B)\right| =0.
\end{align}
Then there exists $\bP^{EP} \in \mathcal{M}_1(\Omega)$ invariant under $(\eta_t^{EP})$ satisfying $\bP^{EP} = \bP$ on $\mathcal{F}^{\infty}_{\infty}$. 

If $(X_t)$ is elliptic in the time direction and $\bP$ is ergodic in the time direction, 
then $\bP^{EP}$ is ergodic with respect to $(\eta_t^{EP})$. Moreover, for any $\mathbb{Q} \ll \bP$ on $\mathcal{F}^{\infty}_{\infty}$,
 \begin{align}\label{eq converge ergodic averages}
 \frac{1}{t} \sum_{s=0}^{t-1} P_{\mathbb{Q}}(\eta_s^{EP}\in \cdot) \text{ converges weakly towards } \bP^{EP} \text{ as } t \rightarrow \infty.
 \end{align}
\end{thm}
 
 \begin{rem}
 There is a certain freedom in the ellipticity and the ergodicity assumptions in Theorem \ref{def assump1}. For instance, the statement still holds if, for some $k \in \nat$, the walker has a positive probability to return to $o$ after $k$ time steps, uniformly in the environment. The definitions can also be modified to require ellipticity and ergodicity with respect other directions $(y,1) \in \Lat$, with  $y \in \mathcal{R}$ (instead of in the direction $(o,1)$. On the other hand, both ellipticity and ergodicity in the time direction are natural assumptions if $\bP$ is the path measure of some stochastic process.
 \end{rem}
 
\begin{cor}[Law of large numbers]\label{thm new lln}
Assume that  $\bP \in \mathcal{M}_1(\Omega)$  is ergodic in the time direction and satisfies \eqref{eq cone mixing Florian}, and that  $(X_t)$ is elliptic in the time direction. 
Then there exists $v\in \reals^d$ such that
$\lim_{t \rightarrow \infty} \frac{1}{t} X_t = v$,  $P_{\bP}-a.s.$
\end{cor}

 Condition \eqref{eq cone mixing Florian} is a considerably weaker mixing assumption than the cone mixing condition introduced by \citet{CometsZeitouniLLNforRWME2004} (see Condition $\mathcal{A}_1$ therein) and used in \citet*{AvenaHollanderRedigRWDRELLN2011} in the context of random walks in dynamic random environment.
For comparison, note that  cone mixing  is equivalent to taking the supremum over events $A \in \mathcal{F}_{< 0} := \mathcal{F}_{\mathbb{Z}^{d+1} \setminus \bH}$ in  \eqref{eq cone mixing Florian}. That Condition \eqref{eq cone mixing Florian} is strictly weaker can already be seen in the case where $\bP$ is i.i.d.\ with respect to space; see Theorem \ref{thm iid space}. Further examples where Condition \eqref{eq cone mixing Florian} improve on the classical cone mixing condition are given in Section \ref{sec examples} and include dynamic random environments with non-uniform mixing properties.


Under a slightly stronger mixing assumption on the environment we obtain more information about $\bP^{EP}$. For this, denote by $\Lambda(l):= 
\{ x \in \bH \colon \norm{x}_1\geq l \}$, $l \in  \nat$, where $\norm{\cdot}_1$ denotes the $l_1$ distance from $(o,0)$, and let $\mathcal{F}_{\geq 0}^{\infty}:= \bigcap_{l\in \nat} \mathcal{F}_{\Lambda(l)}$ be the tail-$\sigma$-algebra with respect to $\mathcal{F}_{\geq0}$.

\begin{thm}[Absolute continuity]\label{def assump2}
Let $\phi \colon \nat \rightarrow [0,1]$ be such that 
\begin{equation}\label{eq bcm}
\sup_{B \in \mathcal{F}_{\Lambda(l)}} \sup_{A \in \mathcal{A}_{-\infty}} \left| \bP( B \mid A) - \bP( B)\right| \leq \phi(l), \end{equation}
 with $\lim_{l \rightarrow \infty} \phi(l)=0$. Then $\bP^{EP}  = \bP$ on $\mathcal{F}_{\geq 0}^{\infty}$ (with $\bP^{EP}$ as in Theorem \ref{def assump1}) and
 \begin{align}\label{eq eqfo in thm}
\sup_{B \in \mathcal{F}_{\Lambda(l)}}  |\bP(B) - \bP^{EP}(B)| \leq \phi(l). 
 \end{align}
 Furthermore, if $(X_t)$ in addition  is elliptic, then $\bP$ and $\bP^{EP}$ are mutually absolutely continuous on $(\Omega,\mathcal{F}_{\geq0})$.
 \end{thm}

Knowing that the environment process converges toward an ergodic measure, it is well known how to apply martingale technics in order to deduce an annealed functional central limit theorem.  However, it may happen that the covariance matrix is trivial. In \citet{RedigVolleringRWDRE2013} it was shown that the covariance matrix is non-trivial in a rather general setting when the environment is given by a Markov process satisfying a certain uniform mixing assumption. It is an interesting question whether $(X_t)$ satisfies an annealed functional central limit theorem with non-trivial covariance matrix under the weaker mixing assumption of \eqref{eq cone mixing Florian}.
 
To obtain a quenched central limit theorem is a much harder problem and is only known in a few cases for random walks in dynamic random environment, see e.g.\ \citet{BricmontKuipainenRWSTME2009},  \citet*{DeuschelGuoRamirezRWBDRE2015}, \citet*{DolgopyatLiveraniNPRWME2008} and \citet*{DolgopyatKellerLIveraniRWME2008}. 
In \cite{DolgopyatKellerLIveraniRWME2008}, Theorem 1, a quenched central limit theorem was proven under technical conditions on both the environment and the environment process. One important condition there was that the environment process has an invariant measure mutually continuous with respect to the invariant measure of the environment. By Theorem \ref{def assump2} above this condition is fulfilled. Combining this result with rate of convergence estimates obtained in \cite{RedigVolleringRWDRE2013}, we conclude a quenched central limit theorem for a large class of uniformly mixing environments. 
 
\begin{cor}[Quenched central limit theorem]\label{thm qclt}
Assume that $(\eta_t)$  is a Markov chain on $E^{\mathbb{Z}^{d}}$. 
For $\sigma, \omega \in \Omega_0$ let $\widehat{P}_{\sigma,\omega}$ be a coupling of $(\eta_t)$ started from $\sigma,\omega \in \Omega_0$ respectively and 
such that, for some $c,C>0$,
\begin{equation}\label{eq QCLT condition}
\sup_{\sigma,\omega \in \Omega} \widehat{P}_{\sigma,\omega}(\eta_t^{(1)}(o) \neq \eta_t^{(2)}(o) ) \leq Ce^{-ct}.
\end{equation}
Furthermore, assume that $(\eta_t)$ satisfies Conditions (A3)-(A4) in \cite{DolgopyatKellerLIveraniRWME2008} and that $(X_t)$ is elliptic. Then, there is a non-trivial $d\times d$ matrix $\Sigma$ such that for $P_{\mu}$-a.e. environment history $(\eta_t)$
\begin{equation}
\frac{X_N -Nv}{\sqrt{N}} \text{ converges weakly towards } \mathcal{N}(0,\Sigma) \quad
\Prob_{(\eta_t)}-a.s., 
 \end{equation}
where $\mu \in \mathcal{M}_1(\Omega_0)$ is the unique ergodic measure with respect to $(\eta_t)$.
\end{cor}

Conditions (A3)-(A4)  in \cite{DolgopyatKellerLIveraniRWME2008} are mixing assumptions on the dynamic random environment $(\eta_t)$. Condition (A3) is a (weak) mixing assumption on $\mu$, whereas Condition (A4) ensures that $(\eta_t)$ is ``local''.  
For the precise definitions we refer to \cite{DolgopyatKellerLIveraniRWME2008}, page 1681.

In \cite{DolgopyatKellerLIveraniRWME2008}, Theorem 2, the statement of Corollary \ref{thm qclt} was proven in a perturbative regime. Corollary \ref{thm qclt} extends their result as there are no restrictions (other than ellipticity) on the transition probabilities of the random walk. We expect that Corollary \ref{thm qclt} can be further improved to a functional CLT assuming only a polynomial decay in \eqref{eq QCLT condition}.

\section{Examples and applications}\label{sec examples}

In this section we present examples of environments which satisfy the conditions of Theorem \ref{def assump1} and Theorem \ref{def assump2}. Particular emphasise is put on environments associated to a hidden Markov model for which we can improve on the necessary mixing assumptions.

\subsection{Environments i.i.d.\ in space}\label{sec iid space}

The influence of the dimension on required mixing speeds is somewhat subtle.
On the one hand, the random walk observes only a local area, and, in the case
of conservative particle systems like the exclusion process, one can expect
that in high dimensions information about observed particles in the past
diffuses away. On the other hand, the higher dimension, the more sites the
random walk can potentially visit in a fixed time. Furthermore, a comparison
with a contact process or directed percolation gives an argument that
information can spread easier in higher dimensions, hence observations along
the path of the random walk could have more influence on future observations
if the dimension increases.

This problem becomes significantly easier when the environment is assumed to
be i.i.d. in space, that is  $\bP = \bigtimes_{x \in \mathbb{Z}^{d}} \bP_o$, and $\bP_o
\in \mathcal{M}_1(E^{\integers})$
is the law of $(\eta_t(x))_{t\in\bZ}$ for any $x\in\bZ^d$.
\begin{thm}\label{thm iid space}
Assume that $\bP = \bigtimes_{x\in\bZ^d} \bP_o$ and that
\begin{align}\label{eq iid space}
\sum_{t \geq 1} \sup_{B \in \mathcal{G}_{\geq t}, A\in \mathcal{G}_{< 0}}
\left| \bP_o(B \mid A ) - \bP_o(B)\right| <\infty ,
\end{align}
where $\mathcal{G}_{\geq t}$ ($\mathcal{G}_{<0}$) is the $\sigma$-algebra 
of $E^\bZ$ generated by the values after time $t$ (before time 0)
with respect to $\bP_o$.
Then \eqref{eq bcm} holds.
\end{thm}

Observe that \eqref{eq iid space} does not depend on the dimension. This is in contrast to the cone mixing condition of \citet{CometsZeitouniLLNforRWME2004}, where an additional factor $t^d$ inside the sum of \eqref{eq iid space} is required. In Subsection \ref{sec examples slow mixing} we present a class of environments which have arbitrary slow polynomial mixing, thus showing that Theorem \ref{thm iid space} yields an essential improvement.

\subsection{Hidden Markov models}\label{sec HMM}
When $\bP$ is the path measure of a stochastic process $(\eta_t)$ evolving on $\Omega_0$, the results of Subsection \ref{sec results Florian} can be improved.
In this subsection we discuss in detail the case where the random environment is governed by a hidden Markov model.

The environment $(\eta_t)$ is a hidden Markov model if it is given via a function of a Markov chain $(\xi_t)$. To be more precise, let $\tE$ be a Polish space, $\tOmega_0=\tE^{\bZ^{\td}}$ with $\td\geq d$, and $\tOmega=\tOmega_0^\bZ$. Denote by $\tilde{\mathcal{F}}$ the corresponding $\sigma$-algebra. 
We assume that the Markov chain $(\xi_t)$ is defined on $\tOmega$ with law $\tP_\xi$ and is ergodic with law $\tmu \in \mathcal{M}_1(\tOmega_0)$. Here $\xi \in \tOmega_0$ denotes the starting configuration. Let $\Phi:\tOmega_0\to\Omega_0=E^{\bZ^d}$ be a translation invariant map and let  $\eta_t=\Phi(\xi_t)$. We call $(\eta_t)$ a \emph{hidden Markov model}, which has $\mu$ as the induced measure on $\Omega_0$ as invariant measure.
We assume throughout that $\Phi$ is of finite range, that is, the function $\Phi(\cdot)(o)$ is $\tilde{\mathcal{F}}_{\Lambda}$-measurable for some $\Lambda \subset \integers^{\tilde{d}}$ finite.

\begin{rem}\label{rem hmm map}
When $\tilde{E}$ is finite, the canonical choice of $\Phi$ is the identity map. However, our setup opens for more sophisticated choices. One example is the projection map. For instance, if $\tilde{d}>1$ and $d=1$, one can consider the hidden Markov model given by $\eta_t(x) = \xi_t(x,0,\dots,0)$. In other words, the random walk only observes the environment in one coordinate.
\end{rem}

Condition \eqref{eq bcm} in Theorem \ref{def assump2} is an infinite volume condition which can be hard to verify by direct computation. The next result yields a sufficient condition which only needs to be checked for single site events. For its statement, we first introduce the concept of $\bP\in \mathcal{M}_1(\Omega)$ having \emph{finite speed of propagation}.

\begin{defn}
We say that $\bP \in \mathcal{M}_1(\Omega)$ has finite speed of propagation if the following holds: for some $\alpha>0$, and for each  $A\in \mathcal{F}_{< 0}$ and $A' \in \mathcal{F}_{\Lambda(\alpha t,t)}$, where $\Lambda(\alpha t,t) := \{ (x,s) \in \bH \colon \norm{ x }_1 \geq \alpha t, 0 < s \leq t \}$,  there is a coupling $\hP_{A,A'}$  of $\bP(\cdot \mid A,A')$ and $\bP(\cdot \mid A)$ such that
\begin{align}\label{eq finite speed of propagation}
\sum_{t \geq 1} t^d \sup_{A \in \mathcal{F}_{<0}, A' \in \mathcal{F}_{\Lambda(\alpha t,t)}}  \hP_{A,A'} \left(\eta_t^1(o) \neq \eta_t^2(o) \right) <\infty.
\end{align}
Furthermore, any such coupling satisfies $\hP_{A,A'}(\cdot) = \hP_{\theta_{x,s}A, \theta_{x,s}A'}(\theta_{x,s} \cdot)$ for all $(x,s) \in \mathbb{Z}^{d+1}$, where $\omega \in \theta_{x,s} A$ if and only if $\theta_{-x,-s} \omega \in A$.
\end{defn}

Finite speed of propagation is a natural assumption for many physical applications. Note that, for many interacting particle systems there is a canonical coupling given by the so-called graphical representation coupling.

\begin{cor}\label{thm:HMM-2}
Assume that $(\xi_t)$ has finite speed of propagation and that
\begin{align}\label{eq hhm estimate}
\sum_{t \geq 1} t^d \sup_{A \in \mathcal{A}_{-\infty}^{-t}} \hP_{\Omega, A} \left( \Phi(\xi^1)_0(o) \neq \Phi(\xi^2)_0(o) \right)< \infty.
\end{align}
Then \eqref{eq bcm} holds for $(\eta_t) = \Phi(\xi_t)$.
\end{cor}

\begin{rem}
The measure $\hP_{\Omega,A}$ denotes the coupling of $\bP(\cdot)$ and $\bP(\cdot \mid A)$.
\end{rem}

Corollary \ref{thm:HMM-2} follows by a  slightly more general statement, see Theorem \ref{thm fsp}. This approach can also be used in cases where $(\xi_t)$ does not have finite speed of propagation. In such cases, \eqref{eq hhm estimate} is sufficient for \eqref{eq cone mixing Florian} to hold.
Observe also that, by applying the projection map introduced in Remark \ref{rem hmm map}, the dimensionality dependence in Condition \eqref{eq hhm estimate} can be replaced by the dimensionality of the range of the random walk. 

\subsubsection*{Markovian environment}
If $\bP$  is the path measure of a Markov chain $(\eta_t)$, we can weaken the mixing assumption. In such cases, we consider $\alpha$ as a function from $\Omega_0 \times \mathbb{Z}^{d}$. 
 Because the Markov property allows us to look at the invariant measure  of the environment process just at a time $0$ instead of in the entire upper half-space $\bH$, we have the following mixing condition. Here we denote by $\mathcal{F}_{=0}^{\infty}$ the tail-$\sigma$-algebra of $\mathcal{F}_{=0} := \mathcal{F}_{\mathbb{Z}^{d} \times \{0\}}$.

\begin{thm}\label{thm markov 1}
Assume that $(\eta_t)$ is a Markov chain with ergodic invariant measure $\mu \in \mathcal{M}_1(\Omega_0)$.
Further, assume that the path measure $\bP_{\mu} \in \mathcal{M}_1(\Omega)$ has finite speed of propagation and that
\begin{align}\label{eq t to the d-1}
\sum_{t \geq 1} t^{d-1} \sup_{A \in \mathcal{A}_{-\infty}^{-t}} \hP_{\Omega,A}\left( \eta_0^1(o)  \neq \eta_0^2(o) \right) < \infty.
\end{align}
Then there exists  $\mu^{EP} \in \mathcal{M}_1(\Omega_0)$ invariant for the Markov chain $(\eta_t^{EP})$  such that  $\mu^{EP}$ agrees with $\mu$ on $\mathcal{F}^{\infty}_{=0}$. 
If in addition $(X_t)$ is elliptic, then $\mu^{EP}$ and $\mu$ are mutually absolutely continuous and $\mu^{EP}$ is ergodic with respect to $(\eta_t^{EP})$.
\end{thm}

It is important to note that \eqref{eq t to the d-1} (as well as \eqref{eq hhm estimate}) does not require $(\eta_t)$ to be uniquely ergodic. However, if for every $\sigma,\xi\in \Omega_0$, there is a coupling $\hP_{\sigma,\xi}$ of $\bP_\sigma$ and $\bP_\xi$ which satisfies the finite speed of propagation property and
\begin{align}\label{eq improve1a}
\sum_{t=1}^\infty t^{d-1} \sup_{\sigma,\xi\in\tOmega_0} \hP_{\sigma,\xi}(\eta_t^1(o)\neq\eta_t^2(o))<\infty,
\end{align}
then it follows, under the assumptions of Theorem \ref{thm markov 1}, that $(\eta_t^{EP})$ is uniquely ergodic. 

Equation \eqref{eq improve1a} should be compared with Assumption 1a in \cite{RedigVolleringRWDRE2013}, that is;
\begin{equation}\label{eq 1a}
\int_{0}^{\infty} t^{(d)} \sup_{\eta,\xi} \hat{\mathbb{E}}_{\eta,\xi} \rho(\eta_t^1(o),\eta_t^2(o)) dt < \infty,
\end{equation}
where $\rho \colon E \times E \rightarrow [0,1]$ is the distance function.
Their  assumption was used to show (among others) the existence of $\mu^{EP} \in \mathcal{M}_1(\Omega_0)$ invariant and ergodic for the Markov chain $(\eta_t^{EP})$, see Lemma 3.2 therein. 
Note in particular that Assumption \eqref{eq 1a} has $t^{(d)}$ inside the integral, whereas \eqref{eq improve1a} only requires $t^{(d-1)}$. 

\subsection{Polynomially mixing environments}\label{sec examples slow mixing}
As example of environments which fully utilise the polynomial mixing assumption of Theorem \ref{thm iid space} and Corollary \ref{thm:HMM-2}, we consider layered environments. These were already considered in \cite{RedigVolleringRWDRE2013} for the same purpose, but since we are in a different setting we use the setting of hidden Markov models.

The idea of layered environments is that, given a summable sequence $(b_n)\subset(0,1)$, for each layer $n$, the process $(\xi_t(\cdot,n))_{t\in \integers_{\geq 0}}$ is an uniform exponentially mixing Markov chain on $[-1,1]$ with an exponential relaxation rate $b_n$, and independent layers. For simplicity, in this example, we choose $\xi_t(\cdot,n)$ to be i.i.d. spin flips, that is, for each $x \in \mathbb{Z}^{\tilde{d}}$,
\begin{align}
\xi_{t+1}(x,n)&=\begin{cases}
\xi_t(x,n),\quad&\text{with probability}\quad 1- b_n;\\
\text{Unif}[-1,1],\quad&\text{with probability}\quad b_n;
\end{cases}
\end{align}
independent for all $x,n,t$. In other words, at each time step the spin retains its old value with probability $1-b_n$ and chooses uniformly on $[-1,1]$ with probability $b_n$.

In the context of the previous subsection we thus have $\tE:=[-1,1]^\bN$. We further choose $\td=d\geq 1$, $E=\{0,1\}$ and set, for a summable sequence $(a_n) \subset(0,1)$,
\begin{align} \Phi(\xi)(x)=\ind_{\sum_{n=1}^\infty a_n \xi(x,n)>0}.\end{align}
The behaviour of this kind of processes is then determined by the two sequences $(a_n)$ and $(b_n)$. 
When $a_n=\frac12 n^{-\alpha}$, $b_n=\frac12 n^{-\beta}$ for some $\alpha,\beta>1$, we have the following bound on the mixing of $(\eta_t)$.

\begin{thm}\label{thm slow mixing}
There are constants $0<c_1<c_2<\infty$ so that
\begin{align}
c_1 t^{\frac{-\alpha+1}{\beta}} \leq \sup_{\xi,\sigma}\norm{\bP_\xi(\eta_t(0)\in\cdot) - \bP_{\sigma}(\eta_t(0)\in \cdot)}_{TV} \leq c_2 t^{\frac{-\alpha+1}{\beta}} (\log t)^{\frac{\alpha-1}{\beta}}.
\end{align}
Here $\norm{\cdot}_{TV}$ is the total variation distance between the two distributions.
In particular, if $\alpha>\beta+1$, then \eqref{eq bcm} holds.
\end{thm}

\subsection{Independent Ornstein-Uhlenbeck processes}\label{sec OU}
With the approach of environments as hidden Markov models, we can also allow for unbounded state spaces where the environment does not mix uniformly, as long as the random walk transition function is simple enough. Here we choose an underlying environment of independent Ornstein-Uhlenbeck processes $(\xi_t^x)_{t\in \bR}$ for each site $x\in\bZ^d$, and the jump rates depend only on the signs, that is, \[\eta_t(x)=\sign(\xi^x_t):=1-2\ind_{\xi^x_t<0}, \quad t \in \integers.\] 
To state the example more formally, we have $\tilde E =\bR$ and $E=\{-1,1\}$, and 
\begin{align}
d\xi^x_t = -\xi^x_t dt + dW_t^x,
\end{align}
where $(W_t^x)_{t\in\bR}$, $x\in\bZ^d$, are independent two-sided Brownian motions. The stationary measure of $\xi_t^x$ is a normal distribution, and $\tilde \mu$ is the product measure of normal distributions.

\begin{thm}\label{lemma:OU}
Let $(\xi_t)_{t\in\bR}$ be an Ornstein-Uhlenbeck process and $\bP$ the two-sided path measure in stationarity. There are constants $c,C>0$ so that
\begin{align*}
\norm{\bP(\xi_t \in \cdot \;|\; A) - \bP(\xi_t \in \cdot)}_{TV} \leq C e^{-c t}
\end{align*}
for all $t\geq 0$ and any $A$ of the form $A= \{\sign(\xi_{-t_k})=a_k, 1\leq k\leq n\}$, $(t_k)$ increasing sequence with $t_1=0$ and $a_k\in \{-1,1\}$, $n$ arbitrary.
In particular, \eqref{eq bcm} holds for $(\eta_t)$.
\end{thm}

\subsection{The contact process}\label{sec examples CP}

As a second example of an environment with non-uniform space-time correlations and which do not satisfy the cone mixing property of \cite{CometsZeitouniLLNforRWME2004}, we consider the contact process $(\eta_t)$ on $\{0,1\}^{\Latd}$  with infection parameter $\lambda \in (0,\infty)$. 

The contact process is one of the simplest interacting particle systems exhibiting a phase transition. That is,
there is a critical threshold $\lambda_c(d) \in (0, \infty)$, depending on the dimension $d$, such that the following holds: if $\lambda\leq \lambda_c(d)$, then the contact process is uniquely ergodic with the measure concentrating on the configuration where all sites equal to $0$ as invariant measure. On the other hand, for all $\lambda > \lambda_c(d)$, the contact process is not uniquely ergodic. In particular, it has a non-trivial ergodic invariant measure, denoted here by $\bar{\nu}_{\lambda}$, also known as the upper invariant measure. As a general reference, and for a precise description of the contact process, we refer to \citet{LiggettSIS1999}.

Random walks on the contact process have recently been studied by  \citet{HollanderSantosRWCP2013} and \citet{MountfordVaresDCP2013}, where the one-dimensional random walk (i.e.\ on $\integers$) was shown to behave diffusively for all $\lambda > \lambda_c(1)$. See also  \citet{BethuelsenHeydenreichRWAD2015} for some results in general dimensions. 

The next theorem sheds new light on the behaviour of the environment process and the random walk for this model on $\Latd$ with $d\geq2$. In the theorem we make use of the projection map, as introduced in Remark \ref{rem hmm map}. That is, we assume $\tilde{d}\geq 2$ and denote by $(\eta_t) = \phi((\xi_t))$ the projection of $(\xi_t)$ onto the $1$-dimensional lattice such that, for $x \in \integers$ and $t\in \integers$, we set $\eta_t(x) = \xi_t(x,0, \dots, 0)$.

\begin{thm}\label{thm EP Contact Process}
Let $\tilde{d}\geq2$ and let $(\xi_t)$ be the contact process with parameter $\lambda > \lambda_c(\tilde{d})$ started from $\bar{\nu}_{\lambda}$. Further, let $(\eta_t^{EP})$ be the environment process corresponding to the process $(\eta_t) = \phi \left( (\xi_t) \right)$.
Then \eqref{eq bcm} holds for $(\eta_t)$.
\end{thm}

Theorem \ref{thm EP Contact Process} can  be extended to higher dimensional projections by following the same approach. The proof strategy of Theorem \ref{thm EP Contact Process} also applies to a larger class of models which satisfy the so-called \emph{downward FKG} property; see Theorem \ref{thm lln FKG}. 

\section{Understanding the environment process}\label{sec expansion}

\subsection{Expansion of the environment process}

In this subsection, we present a key observation for understanding the environment process and for the proofs of Theorem \ref{def assump1} and Theorem \ref{def assump2}.

Intuitively, the distribution of $(\eta_t^{EP})$ should converge to an invariant measure, say $\bP^{EP} \in \mathcal{M}_1(\Omega)$, which describes asymptotic properties. 
To obtain $\bP^{EP}$ and show that it is absolutely continuous with respect to $\bP$, we start by interpreting the law of $\eta_t^{EP}$, $P_{\bP}(\eta_t^{EP} \in \cdot )$, as an approximation.
 With this point of view, $t$ becomes the present time. Going from $t$ to $t+1$ thus means that we look one step further into the past. To reinforce this point of view, we denote by $\bP^{-k} := \bP^{-k}_{\bP} \in \mathcal{M}_1(\Omega \times (\mathbb{Z}^{d})^{\integers_{\geq -k}})$ the joint law of the environment $\bP$ 
 and random walk $(X_t)_{t\geq -k}$ so that $X_0=o$.
  That is, for $k \in \nat$,
  \begin{align}
  \bP^{-k}\left( (\eta,X) \in (B_1,B_2) \right) := P_{\bP} \left( \eta_{k+\cdot}^{EP} \in B_1, (X_{k+\cdot} - X_k) \in B_2 \right).
  \end{align}
For events $B \in \mathcal{F}$, we use the shorthand notation $\bP^{-k}( B)$ for the probability that $\bP^{-k}( (\eta,X) \in (B, (\mathbb{Z}^{d})^{\integers_{\geq -k}}))$.
 
\begin{thm}[Expansion of the environment process]\label{thm expansion}
 For any $k \geq 1$ and $B \in \mathcal{F}$, 
\begin{align}\label{eq expansion 2.1}
\bP^{-k}( B) = \sum_{\gamma \in \Gamma_k} \sum_{A_{-k}^{-1} \in \mathcal{A}_{-k}^{-1}(\gamma) } \bP \left( B,  A_{-k}^{-1} \right) P \left(X_{-k,\dots,0}=\gamma \mid A_{-k}^{-1} \right),
\end{align}
where  $P \left(X_{-k,\dots,0}=\gamma \mid A_{-k}^{-1} \right) := \prod_{i=-k}^{-1} \alpha(\hat{\sigma}_i, \gamma_{i+1}-\gamma_i)$, and $\hat{\sigma} \in \Omega$ is any environment so that $A_{-k}^{-1}(\gamma,\hat{\sigma}) = A_{-k}^{-1}$.
\end{thm}

\begin{proof}

We can rewrite $\bP^{-k}( B)$ as follows;
\begin{align}
\bP^{-k}( B) =& \sum_{\gamma \in \Gamma_k} \sum_{A_{-k}^{-1} \in \mathcal{A}_{-k}^{-1}(\gamma) } \bP^{-k}( B, \:X_{-k,...,0}=\gamma, A_{-k}^{-1}  )
					\\=&\ \sum_{\gamma \in \Gamma_k} \sum_{A_{-k}^{-1} \in \mathcal{A}_{-k}^{-1}(\gamma) } \left[ \bP^{-k}( B \mid A_{-k}^{-1},  X_{-k,...,0}=\gamma  )\right.
					\\ &\left. \quad \quad \quad \quad \quad \quad \quad \bP^{-k}( X_{-k,...,0}=\gamma , A_{-k}^{-1}  ) \right]. \end{align} 
By definition,
\begin{align}
\bP^{-k}(X_{-k,\dots,0}=\gamma, A_{-k}^{-1}) &= P_{\bP}(X_{0,\dots,k} =\gamma-\gamma_{-k},	\eta_k^{EP} \in A_{-k}^{-1})
\\ &= P_{\bP}( 	X_{0,\dots,k} =\gamma-\gamma_{-k}, \theta_{-\gamma_{-k},k} \eta \in A_{-k}^{-1})
\\ &= P(X_{-k,\dots,0}  = \gamma  \mid A_{-k}^{-1}) \bP( \theta_{-\gamma_{-k},k} \eta \in A_{-k}^{-1})
\\&=P(X_{-k,\dots,0}  = \gamma  \mid A_{-k}^{-1}) \bP( A_{-k}^{-1}),
	\end{align}			
where the last equality holds since first the law of the environment is translation invariant. Similarly,
\begin{align}
 \bP^{-k}( B \mid A_{-k}^{-1},  X_{-k,...,0}=\gamma  ) = \bP(B\mid A_{-k}^{-1}).
\end{align}
\end{proof}

The sum in Expansion \eqref{eq expansion 2.1} represents all the possible pasts of the random walk and  the corresponding observed environments from time $-k$ to $-1$. There are two key features with this expansion. 

First, it separates the contribution to $(\eta_t^{EP})$  of the  random walk from that of the random environment. Indeed, the rightmost term in the sum, i.e.\ $P \left(X_{-k,...,0}=\gamma \mid A_{-k}^{-1} \right)$,  can be calculated directly from the transition probabilities of $(X_t)$. On the other hand, the leftmost term in the sum, i.e.\ 
$\bP \left( B,  A_{-k}^{-1} \right)$, only involves the  random environment $\bP$.

A second key feature of \eqref{eq expansion 2.1}  is that it serves as a (formal) expression for the Radon-Nikodym derivate of $ P_{\bP}(\eta_k^{EP} \in \cdot )$ with respect to $\bP$. 
Indeed, \eqref{eq expansion 2.1} yields that for any $B\in \mathcal{F}$ with $\bP(B)>0$,
\begin{align}
 \frac{P_{\bP}(\eta_k^{EP} \in B)}{\bP (B)} =  \sum_{\gamma \in \Gamma_k} \sum_{A_{-k}^{-1} \in \mathcal{A}_{-k}^{-1}(\gamma) }   \frac{\bP\left(B,  A_{-k}^{-1} \right)}{\bP(B)}   P\left(X_{-k,...,0}=\gamma \mid A_{-k}^{-1} \right).
\end{align}

\subsection{Stability}

It is also of interest to compare the effect of changing the environment $\bP$ or the transition probabilities $\alpha \colon \Omega \times \mathbb{Z}^{d} \rightarrow [0,1]$ on the behaviour of the environment process.
 Our next result gives sufficient conditions for the environment process to be stable with respect to perturbations of both these parameters.  This result follows as another consequence of the expansion in Theorem \ref{thm expansion}.
 
 To state the theorem precisely, denote by $(\bP^n)_{n\geq 1}$  a family of measures on $\mathcal{M}_1(\Omega)$ and let $\left(\alpha_n \colon \Omega \times \mathbb{Z}^{d} \rightarrow [0,1] \right)_{n \geq 1}$  be a collection of transition probabilities. Consider for each $n \in \nat$ the corresponding environment process, $(\eta_t^{EP(n)})$,  and let $\bP^{EP(n)} \in \mathcal{M}_1(\Omega)$ be a measure invariant under $(\eta_t^{EP(n)})$.
 
 \begin{thm}
 \label{thm continuity}
  Assume that the following holds.
 \begin{description}
\item[a)] $\epsilon(n) = \sup_{m > n} \sup_{\eta \in \Omega, y \in \mathbb{Z}^{d}} |\alpha_m(\eta,y)-\alpha_n(\eta,y)| \downarrow 0 \text{ as } n \rightarrow \infty.$
\item[b)] $\bP^n \implies \bP \in \mathcal{M}_1(\Omega)$ weakly as $n \rightarrow \infty$.
\item[c)]  $P_{\bP^n}(\eta_t^{EP(n)} \in \cdot ) \implies \bP^{EP(n)}$  weakly as $t \rightarrow \infty$, uniformly in $n$.
 \end{description}
 Then both $\bP^{EP(n)}$ and $P_{\bP}(\eta_t^{EP} )$ converge weakly towards $\bP^{EP} \in \mathcal{M}_1(\Omega)$. In particular, $\bP^{EP}$ is invariant with respect to $(\eta_t^{EP})$.
 \end{thm}

 Condition c) in Theorem \ref{thm continuity} is a strong uniform assumption. If the $\bP^n$'s  are path measures of Markov chains $(\eta_t^{(n)})$, this condition can be replaced by the assumption that the environment process $(\eta_t^{EP})$, i.e., after taking $n \rightarrow \infty$), is uniquely ergodic. For this, recall notation from Section \ref{sec HMM} and let $\mu^{EP(n)} \in \mathcal{M}_1(\Omega_0)$ be an invariant measure with respect to $(\eta_t^{EP(n)})$.
 
\begin{thm}\label{thm continuity2}
Let $(\eta_t)$ be a Markov chain and assume that the following holds. 
\begin{description}
\item[a)] $\epsilon(n) = \sup_{m > n} \sup_{\eta \in \Omega, y \in \mathbb{Z}^{d}} |\alpha_m(\eta,y)-\alpha_n(\eta,y)| \downarrow 0 \text{ as } n \rightarrow \infty.$
\item[b')] $\bP_{\sigma}^n \implies \bP_{\sigma} \in \mathcal{M}_1(\Omega)$ for every starting configuration $\sigma \in \Omega_0$.
\item[c')] $(\eta_t^{EP})$ is uniquely ergodic with invariant measure $\mu^{EP} \in \mathcal{M}_1(\Omega_0)$.
\end{description}
Then  $\mu^{EP(n)} \implies \mu^{EP}$ weakly.
\end{thm}

\begin{rem}Theorem \ref{thm continuity2} does only require that the limiting process $(\eta_t^{EP})$ is uniquely ergodic. In particular, the processes $(\eta_t^{EP(n)})$ do not  need to be uniquely ergodic. As an example of the latter, one can consider the case where $(\eta_t^{(n)})$ is the contact process with parameter $\lambda(n) \downarrow \lambda_c$ and $\inf_{\eta \in \Omega} \alpha_n(\eta,o)\ \uparrow 1$.
\end{rem}

Theorem \ref{thm continuity2} gives a generalisation of Theorem 3.3 in \cite{RedigVolleringRWDRE2013}. There they showed continuity for the environment process with respect to changes of the transition probabilities of the random walk, assuming  that Assumption 1a therein to hold (which we also stated in \eqref{eq 1a}).
Theorem \ref{thm continuity2} yields a similar continuity result which in addition allow for changes in the dynamics of the environment $(\eta_t)$.  Moreover, unique ergodicity is a weaker assumption than the mixing assumption given by \eqref{eq improve1a}, as we have already seen in Subsection \ref{sec iid space} and Subsection \ref{sec HMM}.

\subsection{Estimating the Radon-Nikodym derivative}\label{sec main RN}

We end this section with an alternative route for proving the existence of an invariant measure for the environment process which is absolutely continuous with respect to the underlying environment. An advantage of this approach is that it implies bounds on the Radon-Nikodym derivative.

\begin{thm}\label{thm alternative EP}
Assume that, for some $M_1 \in (0,\infty)$,
\begin{equation}\label{eq thm alternative EP}
\sup _{B \in \mathcal{F}_{\geq0}} \sup_{A_{l} \in \cA_{-l}^{-1}} \left|\frac{ \Prob(B \mid A_{l})}{ \Prob(B)} - 1\right| \leq M_1, \quad \forall \: l \in \nat.
\end{equation}
Then there is  a 
$\bP^{EP}\in \mathcal{M}_1(\Omega)$, invariant  under $(\eta_t^{EP})$, and  $\bP^{EP} \ll \bP$ on $(\Omega,\mathcal{F}_{\geq0})$. Moreover, the corresponding Radon-Nikodym derivative 
is bounded by $M_1$ in the $L_{\infty}$-norm. 
Furthermore, if for some $M_2\in (0,\infty)$, 
\begin{equation}\label{eq thm alternative EP2}
\sup _{B \in \mathcal{F}_{\geq0}} \sup_{A_{l} \in \cA_{-l}^{-1}} \left|\frac{ \Prob(B)}{ \Prob(B \mid A_{l})} - 1\right| \leq M_2, \quad \forall \: l \in \nat.
\end{equation}
Then $\bP \ll \bP^{EP}$ and the corresponding Radon-Nikodym derivative is bounded by $M_2$ in the $L_{\infty}$-norm.
\end{thm}
\begin{rem}
Mutually absolute continuity can also be shown without requiring \eqref{eq thm alternative EP2} to hold. In particular, if \eqref{eq thm alternative EP} holds and $(X_t)$ is elliptic in the time direction, it can be shown that $\bP \ll \bP^{EP}$. Under these assumptions it can moreover be shown that $\bP^{EP}$ is ergodic and that $\left(t^{-1} \sum_{k=1}^tP_{\bP}(\eta_k^{EP} \in \cdot)\right)_{t \geq 1}$ converges weakly towards $\bP^{EP}$.
\end{rem}

Mixing assumption of the type \eqref{eq thm alternative EP} and  \eqref{eq thm alternative EP2}  are typically much stronger than mixing assumptions as in Theorems \ref{def assump1} and \ref{def assump2}.
Nevertheless, we believe that  Theorem \ref{thm alternative EP} is applicable to a wide range of models and is not restricted to the uniform mixing case. 
However, it seems difficult to verify \eqref{eq thm alternative EP} and  \eqref{eq thm alternative EP2} for concrete examples unless strong mixing assumptions are made. 

One class of examples to which Theorem \ref{thm alternative EP}  applies are Gibbs measures in the high-temperature regime satisfying the Dobrushin-Shlosman strong mixing condition (as considered in \citet{RassoulAghaEP2003} for RWRE models); see  Theorem 1.1 (in particular, Condition IIId) in \citet{DobrushinShlosmanCA1985}.
Another class of environments are certain monotone Gibbs measures for which  \citet{AlexanderWeakMixing1998} proved (see Theorems 3.3 and 3.4 therein) that weak mixing implies ratio mixing. In particular, the models considered there satisfy \eqref{eq thm alternative EP} and  \eqref{eq thm alternative EP2}  throughout the uniqueness regime. We also mention the method of disagreement percolation, which is particularly useful for models with hard-core constraints, see \citet{BergMaesDP1994}.

In the case of dynamic random environments  which in addition are reversible with respect to time, typically, the methods described above for random fields can be adapted to yield similar bounds. 
In Section \ref{sec SDP} we introduce a new class of dynamic random environments satisfying \eqref{eq thm alternative EP}, allowing for non-reversible dynamics. We comment next on the scope of this approach.

Our approach 
is by means of disagreement percolation and applies to \emph{discrete-time finite-range Markov chain} $(\eta_t)$. In fact, we shall need  more than subcriticality of the ordinary disagreement process. For what we believe to be technical reasons, we will introduce what we call the \emph{strong disagreement percolation coupling}. This is a  triple $(\eta^1_t,\eta^2_t,\xi_t)$ where $(\eta_t^1,\eta_t^2)$ is a coupling of $\Prob_{\eta_0^1}$ and $\Prob_{\eta_o^2}$, $\xi_t(x)=0$ implies $\eta^1_t(x)=\eta^2_t(x)$, and $\eta^1$ and $\xi$ are independent. That is, the disagreement process $\xi$ and the process $\eta^1$ are independent. This independence is a stronger assumption than regular disagreement percolation and the strong disagreement percolation process is subcritical for models at ``very high-temperature''.  We refer to Section \ref{sec SDP} for a precise construction of the strong disagreement percolation coupling and a proof of the following theorem.

\begin{thm}[Strong disagreement percolation]\label{thm:sdp}
Suppose the strong disagreement percolation process is subcritical. Then there exists $\delta<1$ and $C>0$ so that for any $B \in \mathcal{F}_{=0}$,
\begin{align}
\sup_{\substack{A_{-(k+1)}^{-1} \in \mathcal{A}_{-(k+1)}^{-1}, A_{-k}^{-1} \in \mathcal{A}_{-k}^{-1}\\  A_{-k}^{-1} \subset A_{-(k+1)}^{-1} }} \sup_{B \in \mathcal{F}_{=0}} \abs {\left[\frac{\bP \left( B \; \middle|\; A_{-k-1}^{-1} \right)}{\bP \left( B \; \middle|\; A_{-k}^{-1} \right)}\right]^{\pm1} -1} \leq C\delta^k.
\end{align}
\end{thm}

Theorem \ref{thm:sdp} implies that the environment process $(\eta_t^{EP})$ has a unique invariant distribution, $\mu^{EP} \in \mathcal{M}_1(\Omega_0)$. In particular, $\mu^{EP} $ is absolutely continuous with respect to  the (necessarily unique) invariant measure of $(\eta_t)$, denoted by $\mu \in \mathcal{M}_1(\Omega_0)$.
As a further consequence, we obtain uniform control on the Radon-Nikodym derivative.

\begin{cor}[Uniform control on the Radon-Nikodym derivative]\label{cor SDP}
Assume that the environment $(\eta_t)$ has a strong disagreement percolation coupling which is subcritical. Then $\mu^{EP}$ and $\mu$  are mutually absolutely continuous. Moreover, there exists a constant $M \in (0,\infty)$, depending only on the environment, such that $\norm{\frac{d\mu^{EP}}{d\mu}}_\infty\leq M$ and $\norm{\frac{d\mu}{d\mu^{EP}}}_\infty\leq M$ . \end{cor}

Subcriticality of the strong disagreement coupling is a much stronger assumption than the uniform mixing assumption in \eqref{eq improve1a}. For comparison with other coupling methods, consider for concreteness   the stochastic Ising model  with inverse temperature $\beta>0$ (see e.g.\ \cite{ChazottesRedigVolleringPIMRF2011} for a definition). This model satisfies \eqref{eq improve1a} for all $\beta < \beta_c$, where $\beta_c$ is the critical inverse temperature. On the other hand, it has a subcritical strong disagreement coupling whenever
\begin{align}\label{eq high temp}
\beta < \frac{1}{4d} \log \big(\frac{2d}{2d-1}\big) < \beta_c.
\end{align}
For comparison, this condition is better (with a factor $2$) compared with the disagreement percolation coupling introduced in \cite{ChazottesRedigVolleringPIMRF2011} (see Equation (11) therein).

\begin{rem}
The estimate in \eqref{eq high temp} is  valid for antiferromagnetic models and models with a magnetic field, as also considered in \cite{ChazottesRedigVolleringPIMRF2011}. In particular, the strong disagreement percolation method is not restricted to monotone environments.
\end{rem}

\section{Proofs}\label{sec proofs Florian}

In this section, we present the proofs of the theorems given in the previous sections. In Subsection \ref{sec proofs main results} we give the proofs of theorems in Section \ref{sec results Florian}. Proofs of theorems in Section \ref{sec examples} are given in Subsection \ref{sec proofs examples}. 
In the remaining subsections we present proofs of theorems from Section \ref{sec expansion}. In particular, Subsection \ref{sec SDP} introduces the strong disagreement coupling and contains the proof Theorem \ref{thm:sdp}.

\subsection{Proof  of main results}\label{sec proofs main results}
The main application of the expansion in Theorem \ref{thm expansion} for the proofs of Theorems \ref{def assump1} and \ref{def assump2} is the following lemma.

\begin{lem}\label{lem exp main}
Let $\Lambda \subset \mathbb{Z}^{d+1}$. For $B \in \mathcal{F}_{\Lambda}$ and $k \in \nat$ we have that
\begin{align}\label{eq eqfo 1}
|P_{\bP}(\eta_k^{EP} \in B) - \bP(B)|  \leq \sup_{A \in \mathcal{A}_{-\infty}} |\bP( B \mid A ) -\bP(B)|.
\end{align}
\end{lem}

\begin{proof}
Let $l \in \nat$ and consider any $B \in \mathcal{F}_{\Lambda}$.
From Theorem \ref{thm expansion} we have that for every $k \in \nat$,
\begin{align}
&\big| P_{\bP}\left(\eta_k^{EP} \in B\right) - \bP\left(B\right) \big|
\\= &\big| \sum_{\gamma  \in \Gamma_k} \sum_{A_{-k}^{-1} \in \mathcal{A}_{-k}^{-1} (\gamma)} \bP\left( B, A_{-k}^{-1} \right) P \left(X_{-k,\dots,0} = \gamma\mid A_{-k}^{-1} \right) - \bP\left(B\right)\big|
 \\= &\big| \sum_{\gamma  \in \Gamma_k} \sum_{A_{-k}^{-1} \in \mathcal{A}_{-k}^{-1} (\gamma)} \bP\left( B \mid A_{-k}^{-1} \right) \bP^{-k} \left(X_{-k,\dots,0} = \gamma, A_{-k}^{-1} \right) - \bP\left(B\right)\big|
\\ \leq &\sum_{\gamma  \in \Gamma_k} \sum_{A_{-k}^{-1} \in \mathcal{A}_{-k}^{-1} (\gamma)} \big| \bP\left( B \mid A_{-k}^{-1} \right) -\bP\left(B\right)\big| \bP^{-k} \left(X_{-k,\dots,0} = \gamma, A_{-k}^{-1} \right)
\\ \leq &\sup_{A \in \mathcal{A}_{-\infty}} \big|\bP\left( B \mid A \right) -\bP\left(B\right)\big|,
\end{align}
where in the second last and the last inequality we used the fact that 
\begin{align}
\sum_{\gamma  \in \Gamma_k}  \sum_{A_{-k}^{-1} \in \mathcal{A}_{-k}^{-1} (\gamma)}  \bP^{-k} \left(X_{-k,\dots,0} = \gamma, A_{-k}^{-1}(\gamma,\sigma) \right) =1.
\end{align}
\end{proof}

\begin{proof}[Proof of Theorem \ref{def assump1}]
Consider a sequence $(t_k)$ and a sequence of measures given by $\mathbb{Q}_k := \frac{1}{t_k} \sum_{t=0}^{t_k-1} P_{\bP}(\eta_{t}^{EP}\in \cdot)$ that converges weakly to $\mathbb{Q} \in \mathcal{M}_1(\Omega)$. By standard compactness arguments such a sequence exists and, moreover, any such limiting measure $\mathbb{Q}$ is invariant for $(\eta_t^{EP})$.
 A proof of the last claim is e.g.\ given in \cite{RassoulAghaEP2003}; see page 1457 in the proof of Theorem 3 therein.
 
 Since \eqref{eq cone mixing Florian} is assumed to hold, it follows by Lemma \ref{lem exp main} that for any  $l \in \nat$ and $B\in \mathcal{F}_{\mathcal{C}(l)}$,
 \begin{equation}\label{eq eqfo help}
| \mathbb{Q}_k(B) - \bP(B)| \leq \phi(l),
 \end{equation}
for some $\phi\colon \nat \rightarrow [0,1]$ such that $\lim_{l \rightarrow \infty} \phi(l)=0$. As this estimate is uniform in $k$, we claim that \eqref{eq eqfo help} also holds when $\mathbb{Q}_k$ is replaced by $\mathbb{Q}$. To see this, 
 consider the space of measures measurable with respect to $(\Omega,\mathcal{F}_{\mathcal{C}(l)})$. The ball of radius $\phi(l)$ around $\bP$ (in the total variation sense) is compact in the topology of weak convergence by the Banach-Alaoglu-Theorem. Here we use that the space is compact, the dual of the continuous bounded functions are finite signed measures equipped with the total variation norm, and the weak convergence of measures is the weak-* convergence in this functional-analytic setting. Since the ball is compact it is closed, and any limit point $\mathbb{Q}$ of the sequence $\mathbb{Q}_k$ is also inside the ball. Hence $|\mathbb{Q}(B)-\bP(B)|\leq \phi(l)$ for any $B\in \mathcal{F}_{\mathcal{C}(l)}$ and consequently, since $\lim_{l \rightarrow \infty}\phi(l)=0$, we have $\mathbb{Q} = \bP$ on $\mathcal{F}_{\infty}^{\infty}$.
 
 We continue with the proof that $\bQ$ is ergodic with respect to $(\eta_t^{EP})$, by following the proof of Theorem 2ii) in \cite{RassoulAghaEP2003}.  Denote by $\mathcal{I} \subset \mathcal{F}$ the $\sigma$-algebra consisting of those events invariant under the evolution of $(\eta_t^{EP})$. Further, let $f$ be any local bounded function on $\Omega$ and define $g=\mathbb{E}_{\bQ} \left(f \mid \mathcal{I} \right)$.
 Birkhoffs ergodic theorem implies that, 
 \begin{align}\label{eq Birkhoff1}
P_{\eta} \left( \lim_{n \rightarrow \infty} n^{-1} \sum_{m=1}^n f(\eta_m^{EP}) = g(\eta) \right) =1 ,\quad \text{for }\bQ\text{-a.e. }\eta \in \Omega.
 \end{align}
Using that $\bQ$ is invariant and that $g$ is harmonic, we have
\begin{align}
&\sum_{|x| \leq R} \int \alpha(\eta,x) (g(\eta) - g(\theta_{x,1} \eta))^2 \bQ(d\eta)
\\ = &\int g^2(\eta) \bQ(d\eta) - 2 \int g(\eta) \sum_{|x|\leq R} \alpha(\eta,x) g(\theta_{x,1}\eta) \bQ(d\eta) 
\\ &\quad \quad + \int \sum_{|x|\leq R} \alpha(\eta,x) (g(\theta_{x,1}\eta))^2 \bQ(d\eta)
\\ = & 0.
\end{align}
In particular, since $(X_t)$ is elliptic in the time direction, $g=g \circ \theta_{o,1}$, $\bQ$-a.s. 

Next, for each $t \in \nat$, denote by $B_t \subset \{X_i=o \text{ for all }i \in \{0,\dots, t\}\}$ the event that the random walk does not move in the time-interval $[0,t]$, irrespectively of the environment. Since $(X_t)$ is elliptic in the time direction, $B_t$ has strictly positive probability and can be taken independently of the environment. Further, define
\begin{align}
\bar{g}(\eta) := \limsup_{n \rightarrow \infty} \frac{1}{\bQ(B_t)} \int_{B_t} n^{-1} \sum_{m=1}^n f(\eta_m^{EP}) dP_{\eta}.
\end{align}
 Then, because of \eqref{eq Birkhoff1}, we know that $g=\bar{g}$, $\bQ$-a.s. Further, using the above mentioned independence property, and by possibly taking $t$ large, we note that $\bar{g}$ is $\mathcal{C}(k)$-measurable for any $k\in \nat$.  Consequently, the same holds for $g$, and hence $g$ is $\mathcal{F}_{\infty}^{\infty}$-measurable. Furthermore, since $\bQ = \bP$ on $\mathcal{F}_{\infty}^{\infty}$, this implies that \eqref{eq Birkhoff1} holds $\bP$-a.s., and that $g=g \circ \theta_{o,1}$, $\bP$-a.s. As $\bP$ is ergodic with respect to $\theta_{o,1}$, it  moreover follows that $g$ is constant $\bP$-a.s., and hence also $\bQ$-a.s. Since $f$ was an arbitrary local bounded function, we conclude from this that $\mathcal{I}$ is trivial and thus that $\bQ$ is ergodic with respect to $(\eta_t^{EP})$.

 To conclude the proof we also note that \eqref{eq converge ergodic averages} holds. Indeed, since $\bQ$ was an arbitrary (sub) sequence of $(\bQ_k)$, all the estimates above are valid for any such limiting measure. In particular, each of these limiting measures equal $\bP$ on $\mathcal{F}_{\infty}^{\infty}$, and consequently, they are all ergodic and equal on $\mathcal{I}$. Thus, they are the same, and we conclude that \eqref{eq converge ergodic averages} holds with respect to $\bP$, where we call the limiting measure $\bP^{EP}$. 
Initialising $(\eta_t^{EP})$ with any other probability measure, absolute continuous with respect to $\bP$ on $\mathcal{F}_{\infty}^{\infty}$, the exact same argument as outlined above applies, from which we conclude  \eqref{eq converge ergodic averages} and the proof of Theorem \ref{def assump1}.
  \end{proof}

 \begin{proof}[Proof of Corollary \ref{thm new lln}]
 The claim is an (almost direct) application of ergodicity and that $\bP=\bP^{EP}$ on $\mathcal{F}_{\infty}^{\infty}$. 
Indeed, let $D(\eta) := \sum_{z\in \mathbb{Z}^{d}} z \alpha(\eta,z)$ be the local drift of the random walker in environment $\eta$. A direct consequence of Theorem \ref{def assump1} is that
\begin{align}\label{eq help help help}
P_{\bP^{EP}} \left(\lim_{n\rightarrow \infty}n^{-1} \sum_{k=1}^n D(\theta_{X_k,k}\eta) = v \right) =1,
\end{align}
where $v= \int D(\eta) P_{\bP^{EP}}(d\eta)$. By using that $\bP = \bP^{EP}$ on $\mathcal{F}_{\infty}^{\infty}$, it follows that this also holds with respect to $P_{\bP}$. 
 Now, note that $M_n = X_n - \sum_{m=0}^{n-1} D(\theta_{X_k,k}\eta)$ is a martingale with bounded increments under $P_{\eta}$. Therefore $P_{\eta}(\lim_{n \rightarrow \infty} n^{-1}M_n =0)=1$ which together with \eqref{eq help help help} implies the law of large numbers.
  \end{proof}

We next turn to the proof of Theorem \ref{def assump2}.
The following lemma is essentially copied from \cite{BergerCohenRosenthalLLT2014}.

\begin{lem}\label{lem thm 4.2 help}
Assume $\bP^{EP}$ is invariant with respect to $(\eta_t^{EP})$ and that $\bP^{EP}$ and $\bP$ restricted to $(\Omega, \mathcal{F}_{\geq0})$ are not singular. Assume $(X_t)$ is elliptic. Then there exists  $\bP^{EP}_c \in \mathcal{M}_1(\Omega)$, invariant for $(\eta_t^{EP})$ and mutually absolutely continuous to $\bP$ on $(\Omega,\mathcal{F}_{\geq0})$.
\end{lem}

\begin{proof}[Proof of Lemma \ref{lem thm 4.2 help}]
Consider the (unique) Lebesgue decomposition of $\bP^{EP}$ with respect to $\bP$ restricted to $(\Omega, \mathcal{F}_{\geq0})$. That is, let
\begin{align}
\bP^{EP}(B) = \alpha \bP^{EP}_c(B) + (1-\alpha) \bP^{EP}_s(B), \quad \forall \: B \in \mathcal{F}_{\geq0},
\end{align}
where $\bP^{EP}_c \ll \bP$ and $\bP^{EP}_s \bot \bP$ on $(\Omega, \mathcal{F}_{\geq0})$. By assumption, we know that $\alpha>0$. If $\alpha=1$, the statement is immediate. Thus, assume $\alpha \in (0,1)$.
In a first step, observe that $(\theta_{y,1} \circ \bP^{EP})_{c} = \theta_{y,1} \circ \bP^{EP}_c$ for every $y \in \mathcal{R}$. This follows from translation invariance of $\bP$ which implies that taking the continuous part with respect to $\bP$ is the same as taking the continuous part with respect to $\theta_{y,1} \circ \bP$. The same is true for the singular part $\bP^{EP}_s$.

 Note that, since $E$ is finite and $(X_t)$ is finite range, we have that ellipticity in fact implies uniform ellipticity. That is, there is an $\epsilon>0$ such that
\begin{align}
\inf_{y \in \mathcal{R}} \inf_{\eta \in \Omega} \alpha(\eta,y) \geq \epsilon >0.
\end{align}
In particular, by invariance of $\bP^{EP}$,
\begin{align}
\bP^{EP} = \sum_{y \in \mathcal{R}} \alpha(\cdot,y) \theta_{y,1} \circ \bP^{EP} \geq \epsilon \sum_{y \in \mathcal{R}}  \theta_{y,1} \circ \bP^{EP} 
\end{align}
and therefore $\theta_{y,1} \circ \bP^{EP} \ll \bP^{EP}$ for every $y \in \mathcal{R}$.
By using first ellipticity and then $(\theta_{y,1} \circ \bP^{EP})_{c} = \theta_{y,1} \circ \bP^{EP}_c$ we have
\begin{align}
\bP^{EP}_c &= \left(\sum_{y \in \mathcal{R}} \alpha(\cdot, y) \theta_{y,1} \circ \bP^{EP}\right)_c
= \sum_{y \in \mathcal{R}} \alpha(\cdot, y) \left(\theta_{y,1} \circ \bP^{EP}\right)_c
\\ &= \sum_{y \in \mathcal{R}} \alpha(\cdot, y) \theta_{y,1} \circ \bP^{EP}_c , \label{eq singular 1}
\end{align}
which means that $\bP^{EP}_c$ is invariant for $(\eta_t^{EP})$.

Let $f= \frac{d\bP^{EP}_c}{d\bP}$ and define $B= \{ \eta \in \Omega \colon f(\eta) >0\}$. As a consequence of \eqref{eq singular 1} we have 
\begin{align}
f(\eta) \geq \epsilon \sum_{y \in \mathcal{R}} f(\theta_{y,1} \eta),
\end{align}
and, in particular, $\eta \in B$ implies $\theta_{y,1} \eta \in B$, $y\in \mathcal{R}$. In particular, $B$ is invariant under $\theta_{o,1}$, and by ergodicity of $\bP$ this is a $0-1$ event. Since by assumption $\alpha >0$ we have $\bP(B)=1$ and therefore $\bP \ll \bP^{EP}_c$ on $(\Omega, \mathcal{F}_{\geq0})$.
\end{proof}

 \begin{lem}\label{lem extended expansion}
 Let  $\Lambda \subset \mathbb{Z}^{d+1}$ finite and fix $\sigma \in E^{\Lambda^c}$. Let $\bP(\cdot \mid \sigma) $ and $\bP^{-k}(\cdot \mid \sigma)$ be the regular conditional probabilities of $\bP$ and $\bP^{-k}$ on $E^{\Lambda}$ given $\sigma$.
 Then, for $B \in \mathcal{F}_{\Lambda}$ and $k \geq 1$,
\begin{align}
\bP^{-k}( B\mid \sigma) = \sum_{\gamma \in \Gamma_k} \sum_{A_{-k}^{-1} \in \mathcal{A}_{-k}^{-1}(\gamma) } \bP \left( B, A_{-k}^{-1} \mid  \sigma \right) P \left(X_{-k,\dots,0}=\gamma \mid A_{-k}^{-1} \right).
\end{align}
 \end{lem}
 
 \begin{proof}
 The proof is mostly as for the unconditional expansion. Additionally we use the following equalities:
 \begin{align}
 \bP^{-k}(B \mid A_{-k}^{-1}, X_{-k,\dots, 0}=\gamma, \sigma) &= P_{\bP}( \theta_{-\gamma_{-k},k}B \mid \theta_{-\gamma_{-k},k} A_{-k}^{-1}, \theta_{-\gamma_{-k},k} \sigma)
 \\ &=\bP(B\mid A_{-k}^{-1}, \sigma)
 \end{align}
and 
\begin{align}
&\bP^{-k}( X_{-k,\dots,0} = \gamma, A_{-k}^{-1} \mid \sigma)
\\= &P_{\bP}(\eta_k^{EP} \in A_{-k}^{-1}, X_{0,\dots, k} =  \gamma - \gamma_{-k} 
\mid \theta_{X_k,k} \sigma )
\\ = &P( X_{0,\dots,k} = \gamma - \gamma_{-k} \mid A_{-k}^{-1})
\bP( \theta_{-\gamma_{-k,k}} \eta \in A_{-k}^{-1} \mid \theta_{-\gamma_{-k,k}} \sigma)
\\ =  &P( X_{0,\dots,k} = \gamma - \gamma_{-k} \mid A_{-k}^{-1}) \bP(  \eta \in A_{-k}^{-1} \mid  \sigma).
\end{align}
 Take also note that summation should only include events $A_{-k}^{-1}$ which have positive probability with respect to the conditional law given $\sigma$.
 \end{proof}

  \begin{proof}[Proof of Theorem \ref{def assump2}]
  By applying the same line of reasoning as in the proof of Theorem \ref{def assump1}, it easy to see that  \eqref{eq eqfo in thm} holds as a consequence of \eqref{eq bcm} and Lemma \ref{lem exp main}. In particular, there is a measure $\bQ$ invariant under $(\eta_t^{EP})$ such that \eqref{eq eqfo in thm} holds, and consequently $\bQ = \bP$ on $\mathcal{F}_{\geq 0}^{\infty}$. We focus on the proof that $\bQ$ and $\bP$ are mutually absolutely continuous under the additional assumption that $(X_t)$ is elliptic.

  Since \eqref{eq eqfo in thm} holds, there is an $l \in \nat$ such that $\sup_{B \in \mathcal{F}_{\Lambda(l)}} |\bQ(B)-\bP(B)|<1$. In particular, $\bP$ and $\bQ$ are not singular on  $\mathcal{F}_{\Lambda(l)}$.  In order to conclude that $\bQ$ and $\bP$ are not singular on $\mathcal{F}_{\geq 0}$ we make use of  Lemma \ref{lem extended expansion} and the assumption that $|E|<\infty$. Indeed, for any $\sigma \in E^{\Lambda(l)}$ we have by Lemma \ref{lem extended expansion} that
 $ \bP^{-k}(\cdot \mid \sigma) \ll \bP(\cdot \mid \sigma)$.
  Further, since $E$ is finite any local function is continuous and hence we also have $\bQ(\cdot \mid \sigma) \ll \bP(\cdot \mid \sigma)$. And since $\bQ$ is non-singular on $\Lambda(l)$ with respect to $\bP$ it has non-trivial continuous part and corresponding density on $\Lambda(l)$. Thus, we now also have shown that conditioned on $\Lambda(l)$ the measure $\bQ$ has a density inside $\bH \setminus \Lambda(l)$. It hence follows that $\bQ$ is not singular with respect to $\bP$ on $\bH$. As a consequence of Lemma \ref{lem thm 4.2 help} and that $\bQ$ and $\bP$ are not singular on $\bH$ we conclude that, when $(X_t)$ is elliptic, there is a measure $\bP^{EP} \in \mathcal{M}_1(\Omega)$ invariant under $(\eta_t^{EP})$ such that $\bP^{EP}$ and $\bP$ are mutually absolutely continuous on $\mathcal{F}_{\geq 0}$. 
  
   \end{proof} 
  
\begin{proof}[Proof of Corollary \ref{thm qclt}]
Corollary \ref{thm qclt} is an application of Theorem 1 in \cite{DolgopyatKellerLIveraniRWME2008}.
In order to fulfil the requirements of their theorem six conditions needs to be satisfied, i.e. (A0)-(A5) therein. Our main contribution is that Condition (A1) is satisfied when \eqref{eq QCLT condition} holds. This follows as a direct consequence of Theorem \ref{def assump2}. Furthermore, that Conditions (A0)  holds under \eqref{eq QCLT condition} follows partly by the mixing assumptions on $\mu$. Moreover, that  also $\mu^{EP}$ and $\mu$ are mixing is a consequence of Theorem 3.4 in \cite{RedigVolleringRWDRE2013} which yields exponential rate of convergence for environment process under the assumption that \eqref{eq QCLT condition} holds. By the same reasoning, Theorem 3.4 in \cite{RedigVolleringRWDRE2013} also implies that Condition (A2) holds true. Lastly, we note that Conditions (A3) and (A4) are true by assumption and 
Condition (A5) is satisfied since $(X_t)$ is elliptic.
\end{proof}

 \subsection{Proof of examples}\label{sec proofs examples}

\subsubsection{Proof of Theorem \ref{thm iid space}}

\begin{proof}

For $k \in \nat$, let $\gamma \in \Gamma_k$ and consider $A \in
\mathcal{A}_{-k}^{-1}(\gamma)$. Since $A$ consists of a fixed observation of
the environment along the path $\gamma$ we can write
$A=\bigcap_{x\in\bZ^d}A_x$, where $A_x$ is the observation on the line $\{
(x,s) \colon s \in \integers \}$. Without change of notation we also treat
$A_x$ as an event on the space $E^\bZ$.
Denote by $\hP_x$ the optimal coupling (in the sense of total variation
distance) of $\bP_o(\cdot \mid A_x)$ and $\bP_o(\cdot)$, and by $\hP=
\bigtimes_{x \in \mathbb{Z}^{d}} \hP_x$. The product structure of $\bP$ plus the fact
that $A$ is given by the intersection of the events $A_x$ gives us that $\hP$
is a coupling of $\bP(\cdot \mid A)$ and $\bP(\cdot)$.

For $l \in \nat$, let $B\in \mathcal{F}_{\Lambda(l)}$. 
We  have that 
\begin{align}
\left| \bP(B\mid A) - \bP(B) \right| &\leq \hP\left( \exists (x,s) \in \Lambda(l)
\colon \eta_s^1(x) \neq \eta_s^2(x) \right)
\\ &\leq \sum_{x \in \mathbb{Z}^{d}} \hP\left( \exists s \geq 0 \colon (x,s) \in
\Lambda(l) \text{ and } \eta_s^1(x) \neq \eta_s^2(x) \right)
\\ &\leq \sum_{t <0} \sum_{(x,t) \in \gamma + [-R,R]^d\times \{0\}}
\sup_{B\in\mathcal{G}_{\geq 0\vee (l-|x|)}} \left|\bP_o( B|A_x)-\bP(B)\right|.
\end{align}
The last line follows from the fact that $\eta^1(x)$ and $\eta^2(x)$ can only
differ if the site $x$ is part of the observation $A$, since otherwise
$A_x=E^\bZ$. Condition  \eqref{eq iid space} thus ensures  that the sum in the
last line is finite. In particular, the sum converges to $0$ as $l\rightarrow
\infty$. This shows that $\bP$ satisfies \eqref{eq bcm}.
\end{proof}

\subsubsection{Proof of Corollary \ref{thm:HMM-2}}

Corollary \ref{thm:HMM-2} follows by a slightly stronger statement, which we state and prove first.

\begin{thm}\label{thm fsp}
Assume that $\bP \in \mathcal{M}_1(\Omega)$ has finite speed of propagation and that \begin{align}\label{eq t to the d}
\sum_{t \geq 1} t^d \sup_{A' \in \mathcal{A}_{-\infty}^{-t}} \hP_{\Omega, A'} \left( \eta_0^1(o) \neq \eta_0^2(o) \right)< \infty.
\end{align}
Then $\bP$ satisfies the conditions of Theorem \ref{def assump2}.
\end{thm}

\begin{proof}[Proof of Theorem \ref{thm fsp}]
Let $l \in \nat$ and consider $B \in \mathcal{F}_{\Lambda(l)}$. 
By Theorem \ref{def assump2}, it is sufficient to obtain uniform estimates of the form
$|\bP(B \mid A)-\bP(B)| \leq \phi(l)$, where $A \in \mathcal{A}_{-\infty}$, and where $\phi(l)$ approaches $0$ as $l \rightarrow \infty$.
For this, we first note that
\begin{align}
\left|\bP \left(B \mid A\right) - \bP\left(B\right)\right| &\leq \hP_{A,\Omega}\left( \eta_t^1(x) \neq \eta_t^2(x) \text{ for some } (x,t) \in \Lambda(l) \right)
\\&\leq \sum_{(x,t) \in \Lambda(l)} \hP_{A,\Omega}\left( \eta_t^1(x) \neq \eta_t^2(x) \right).
\end{align}
Thus, it suffices to control $\hP_{A,\Omega}\left( \eta_t^1(x) \neq \eta_t^2(x) \right)$ for each $(x,t) \in \Lambda(l)$. For this, fix $(x,t) \in \Lambda(l)$ such that $\norm{(x,t)}_1 \geq \alpha s$ for some $s \geq0$. 
Further, let $A' \in \mathcal{A}_{-\infty}^{-s}$ be such that $A'\cap A= A$, and denote by $\tilde{\bP}_{A,A',\Omega}$ a measure on $\Omega \times \Omega \times \Omega$ such that 
\begin{align}
&\tilde{\bP}_{A,A',\Omega} (\eta^{1} \in \cdot, \eta^{2} \in \cdot, \Omega) = \hP_{A,A'} ( \eta^1\in \cdot, \eta^2 \in \cdot); 
\\&\tilde{\bP}_{A,A',\Omega} (\eta^{1} \in \cdot, \Omega, \eta^{3} \in \cdot) =  \hP_{A,\Omega} (\eta^1 \in \cdot, \eta^3 \in \cdot);
\\&\tilde{\bP}_{A,A',\Omega}( \Omega, \eta^{2} \in \cdot, \eta^{3} \in \cdot) =  \hP_{A',\Omega}( \eta^2 \in \cdot, \eta^3 \in \cdot).
\end{align}
 We then have that
\begin{align}
\hP_{A,\Omega}\left( \eta_t^1(x) \neq \eta_t^2(x) \right) = &\tilde{\bP}_{A,A',\Omega} (\eta_t^{1}(x)\neq  \eta_t^{3}(x) )
\\ \leq & \tilde{\bP}_{A,A',\Omega} (\eta_t^{1}(x)\neq  \eta_t^{2}(x) \text{ or }  \eta_t^{2}(x)\neq  \eta_t^{3}(x) )
\\ \leq & \tilde{\bP}_{A,A',\Omega} (\eta_t^{1}(x)\neq  \eta_t^{2}(x) ) +  \tilde{\bP}_{A,A',\Omega} (\eta_t^{2}(x)\neq  \eta_t^{3}(x) )
\\ = &  \hP_{A,A'} \left(\eta_t^1(x) \neq \eta_t^2(x) \right) + \hP_{A',\Omega}\left( \eta_t^1(x) \neq \eta_t^2(x) \right),
\end{align}
Furthermore, it holds that
\begin{align}
&\hP_{A,A'} \left(\eta_t^1(x) \neq \eta_t^2(x) \right) + \hP_{\Omega,A'}\left( \eta_t^1(x) \neq \eta_t^2(x) \right) 
\\ \leq &\sup_{A \in \mathcal{F}_{< 0}, A' \in \mathcal{F}_{\Lambda(\alpha s, s)}} \hP_{A,A'} \left(\eta_{s}^1(o) \neq \eta_{s}^2(o) \right)+  \sup_{A'' \in \mathcal{A}_{-\infty}^{-\alpha s}} \hP_{\Omega,A''}\left( \eta_0^1(o) \neq \eta_0^2(o) \right),
\end{align}
since the finite speed of propagation coupling is invariant with respect to translations of the conditioning and the argument.
 Thus, by the analysis above, we obtain that
\begin{align}
\left|\bP(B \mid A) - \bP(B) \right| \leq &\left( \sum_{\substack{(x,t) \in \Lambda(l) \\ \norm{(x,t)}_1 =\alpha s}} \sup_{A \in \mathcal{F}_{< 0}, A' \in \mathcal{F}_{\Lambda(\alpha s, s)}} \hP_{A,A'} \left( \eta_{s}^1(o) \neq \eta_{s}^2(o) \right) \right)\\  +  
&\left( \sum_{\substack{(x,t) \in \Lambda(l) \\ \norm{(x,t)}_1 =\alpha s}}  \sup_{A'' \in \mathcal{A}_{-\infty}^{-\alpha s}} \hP_{A'', \Omega} \left( \eta_0^1(x) \neq \eta_0^2(x) \right) \label{eq fsop last}
\right) .
\end{align}
To conclude the proof, we note that the number of site in $\bH$ at distance $\alpha s$ from the origin is of order $s^d$. Thus, due to \eqref{eq finite speed of propagation} the first sum on the r.h.s.\ of  \eqref{eq fsop last} converges towards $0$ as $l$ approaches $\infty$. Similarly, by applying \eqref{eq t to the d}, also the second sum on the r.h.s.\ of  \eqref{eq fsop last} converges towards $0$ as $l$ approaches $\infty$. From this we conclude the proof.
\end{proof}

\begin{proof}[Proof of Corollary \ref{thm:HMM-2}]
The proof of Corollary \ref{thm:HMM-2} follows along the lines of the proof of Theorem \ref{thm fsp}, by making use of the finite speed of propagation property and \eqref{eq hhm estimate}. First note that, for any $B \in \mathcal{F}_{\Lambda(l)}$ and $A \in \mathcal{A}_{-\infty}$, 
\begin{align}
\left|\bP \left(B \mid A\right) - \bP\left(B\right)\right| &= \left| \tP(\Phi (\xi) \in B \mid \xi \in \Phi^{-1}A ) - \tP ( \Phi(\xi) \in B) \right|
\\ &\leq \hP_{\Omega,\Phi^{-1}A}\left( \eta_t^1(x) \neq \eta_t^2(x) \text{ for some } (x,t) \in \Lambda(l) \right)
\\&\leq \sum_{(x,t) \in \Lambda(l)} \hP_{\Omega,\Phi^{-1}A}\left( \eta_t^1(x) \neq \eta_t^2(x) \right).
\end{align}
Thus, it suffices control  $\hP_{\Omega,\Phi^{-1}A}\left( \eta_t^1(x) \neq \eta_t^2(x) \right)$ for each $(x,t) \in \Lambda(l)$. and to show that the latter term above approaches $0$ as $l \rightarrow \infty$. For this, since $\Phi$ is assumed to be finite range, we note that the finite speed of propagation property $(\xi_t)$ transfers to events of the form $\Phi^{-1} A$. Thus, by considering a coupling $\tilde{\bP}_{\Phi^{-1} A,\Phi^{-1}A',\Omega}$, similar to the coupling in the  proof of Theorem \ref{thm fsp}, and where $A'\in \mathcal{A}_{-\infty}^{-s}$ and $A'\cap A=A$. Further, by applying the estimates  \eqref{eq finite speed of propagation} and  \eqref{eq hhm estimate}, we may proceed by the same line of reasoning as in the proof of Theorem \ref{thm fsp}, from which we conclude the proof.
\end{proof}

\subsubsection{Proof of Theorem \ref{thm markov 1}}

\begin{proof}[Proof of Theorem \ref{thm markov 1}]

Denote by $\mu^{EP}$ any limiting measure of the Cesaro means $\mu^{EP}_n := \frac{1}{n} \sum_{k=1}^n \bP_{\mu} (\eta_k^{EP} \in \cdot ) \in \mathcal{M}_1(\Omega_0)$ (by possibly taking subsequential limits) and note that $\mu^{EP}$ is invariant with respect to $(\eta_t^{EP})$. 

We first show that $\mu^{EP}$ agrees with $\mu$ on $\mathcal{F}_{=0}^{\infty}$.  Let $l \in \nat$ and consider any $B \in \mathcal{F}_{\Lambda_0(l)}$  with $\Lambda_0(l) := \{ (x,0) \colon \norm{(x,0)}_1 \geq l \}$. Similar to the proof of Theorem \ref{thm fsp}, it follows that, for any $A \in \mathcal{A}_{-\infty}$,
\begin{align}
\left|\bP(B \mid A) - \bP(B) \right| \leq &\left( \sum_{x \in \Lambda_0(l)} \sup_{A \in \mathcal{F}_{< 0}, A' \in \mathcal{F}_{\Lambda(\alpha s, s)}} \hP_{A,A'} \left( \eta_{\alpha s}^1(o) \neq \eta_{\alpha s}^2(o) \right) \right)\\  +  
&\left( \sum_{x \in \Lambda_0(l)}  \sup_{A'' \in \mathcal{A}_{-\infty}^{-\alpha s}} \hP_{A'', \Omega} \left( \eta_0^1(o) \neq \eta_0^2(o) \right) 
\right).
\end{align}
Since $\bP$ has finite speed of propagation, the first term converges to $0$ as $l$ approaches $\infty$.
For the second term, note that the number of sites in $\mathbb{Z}^{d}$ at distance $t$ from the origin is of order $t^{d-1}$. Thus, by \eqref{eq t to the d-1}, also the second term converges to $0$ as $l \rightarrow \infty$. This yields that  $\mu^{EP}$ agrees with $\mu$ on $\mathcal{F}_{=0}^{\infty}$, and that $\mu^{EP}$ and $\mu$ are non-singular on $(\Omega,\mathcal{F}_{\Lambda_0(l)})$ for all $l \in \nat$ sufficiently large. 

Next, assume in addition that $(X_t)$ is elliptic. By Lemma \ref{lem extended expansion} and an argument as in the proof of Theorem \ref{def assump2},  we conclude that $\mu$ and $\mu^{EP}$ are non-singular on $(\Omega,\mathcal{F}_{=0})$.  From this, we conclude that there is probability  measure $\hat{\mu}^{EP}$, invariant under $\mu$ and such that $\mu$ and $\hat{\mu}^{EP}$ are mutually absolutely continuous. This follows analogous to the proof of Theorem \ref{def assump2} by making use of (a slight adaptation of) Lemma \ref{lem thm 4.2 help} and the assumption that $(X_t)$ is elliptic. Consequently, the path measure of $(\eta_t^{EP})$ initialised from $\mu^{EP}$, denoted by $\bP^{EP} \in \mathcal{M}_1(\Omega,\mathcal{F})$, is mutually absolutely continuous to $\bP$ on $(\Omega, \mathcal{F}_{\geq0})$. Thus, since ellipticity implies ellipticity in the time direction, and since $\mu$ is ergodic under $(\eta_t)$ we conclude that $\mu^{EP}$ is ergodic under $(\eta_t^{EP})$, as follows similar to the proof of ergodicity in Theorem \ref{def assump1}. 

\end{proof}

\subsubsection{Proof of Theorem \ref{thm slow mixing} }

We next prove that the environments constructed in Subsection \ref{sec examples slow mixing} have arbitrary slow polynomial mixing.

\begin{proof}[Proof of Theorem \ref{thm slow mixing}]
First we will show the upper bound, by choosing a particular coupling.
The natural coupling of $\bP_\xi$ and $\bP_\sigma$ is that $\xi_t^1(x,n)$ and $\xi_t^2(x,n)$ share the resampling events of probability $b_n$, so that after the first resampling, the spins are identical. Note that this coupling can naturally be extended to an arbitrary number of initial configurations. If we denote by $\xi^\sigma_t$ the configuration at time $t$ when started in $\sigma$, we have under this coupling
\begin{align}
\xi^{-{\mathbf 1}}_t(x,n)\leq \xi^{\sigma}_t(x,n) \leq \xi^{+\mathbf 1}_t(x,n)
\end{align}
and hence $\eta^{-{\mathbf 1}}_t(x)\leq \eta^\sigma_t(x)\leq \eta^{+{\mathbf 1}}_t(x)$ for all $t,x,n,\sigma$. In particular it follows that
\begin{align}
\hP_{\xi,\sigma}(\eta^1_t(0)\neq\eta^2_t(0))\leq \hP_{\mathbf 1,-\mathbf 1}(\eta^1_t(0)=1)-\hP_{\mathbf 1,-\mathbf 1}(\eta^2_t(0)=1).
\end{align}
Let $R_t:=\{n\in\bN : \xi^1_t(0,n)= \xi^2_t(0,n) \}$. We have
\begin{align}\label{eq:layered-2}
\bP_{\pm\mathbf 1}(\eta_t(0)=1) &= \hP_{\mathbf 1,-\mathbf 1}\left(\sum_{n\in R_t} a_n \xi^1_t(0,n)> \mp \sum_{n\in R_t^c} a_n\right),
\end{align}
and hence,
\begin{align}\label{eq:layered-1}
\hP_{\xi,\sigma}(\eta^1_t(0)\neq\eta^2_t(0))&\leq \hE \int_{-\sum_{n\in R_t^c} a_n}^{\sum_{n\in R_t^c} a_n}f_{R_t}(x)dx = \hE \sum_{n\in R_t^c} a_n \int_{-1}^{1} f_{R_t}(\sum_{n\in R_t^c} a_n x)dx,
\end{align}
where $f_R$ is the density of $\sum_{n\in R}a_n Y_n$ and $(Y_n)_n$ are i.i.d. uniform $[-1,1]$ distributed. A simple convolution of the individual densities shows that $f_R\leq \min_{n\in R}(2a_n)^{-1}$, hence the above is bounded by
\begin{align}
&\hE \left((\min_{n\in R_t} a_n)^{-1} \sum_{n\in R_t^c}a_n \right) = \sum_{k\in\bN}\hP(\min R_t = k)a_k^{-1}\left(\sum_{n<k}a_n\right)\hE \left(\sum_{n>k} \ind_{n\in R_t^c} a_n\right)	\\
&= \sum_{k\in\bN}\left(\prod_{n<k}(1-b_n)^t\right)\left(1-(1-b_k)^t\right)a_k^{-1}\left(\sum_{n<k}a_n\right) \left(\sum_{n>k} a_n(1-b_n)^t\right). \label{eq:layered-3}
\end{align}
To obtain polynomial decay, we choose $a_n=\frac12 n^{-\alpha}$ and $b_n=\frac12 n^{-\beta}$. Then we can find some constant $C>0$ so that 
\[  \sum_{k\in\bN}\left(\prod_{n<k}(1-b_n)^t\right)\left(1-(1-b_k)^t\right)a_k^{-1}\left(\sum_{n<k}a_n\right)\leq C. \]
With this and $1-b_n\leq e^{-b_n}$,
\begin{align}
\eqref{eq:layered-3}\leq C \sum_{n\in\bN}a_n(1-b_n)^t &\leq \sum_{n\leq (t/\log t^2)^{\frac1\beta}}a_n e^{-b_n t} + \sum_{n> (t / \log t^2)^{\frac1\beta}}a_n e^{-b_n t}	\\
& \leq \sum_{n\leq (t /\log t^2)^{\frac1\beta}}a_n t^{-\log t} + \sum_{n> (t / \log t^2)^{\frac1\beta}}a_n 	\\
&\leq c_2 t^{\frac{-\alpha+1}{\beta}} (\log t)^{\frac{\alpha-1}{\beta}}.\label{eq:layered-4}
\end{align}

For a lower bound, we use \eqref{eq:layered-2} plus the fact that \eqref{eq:layered-1} is an equality for $\xi=+\mathbf 1$ and $\sigma=\mathbf{-1}$, so that we have 
\begin{align}
&\sup_{\xi,\sigma}\norm{\bP_\xi(\eta_t(0)\in\cdot) - \bP_{\sigma}(\eta_t(0)\in \cdot)}_{TV} = \hE \sum_{n\in R_t^c} a_n \int_{-1}^{1} f_{R_t}(\sum_{n\in R_t^c} a_n x)dx .
\end{align}
The density $f_R$ has a unique local maximum at 0 and its support is $[-\sum_{n\in R}a_n,\sum_{n\in R}a_n]$, so that we can lower bound the integral by replacing $f_{R_t}$ with $(2\sum_{n\in R_t}a_n)^{-1}$:
\begin{align}
 \eqref{eq:layered-4}&\geq \hE_{\mathbf 1, -\mathbf 1} ( \sum_{n\in R_t} a_n )^{-1} {\sum_{n\in R_t^c} a_n}  
 \geq ( \sum_{n\in \bN} a_n )^{-1} \sum_{n\in \bN} a_n (1-b_n)^t \\
 &\geq ( \sum_{n\in \bN} a_n )^{-1} \sum_{n \geq t^{-\beta}} a_n (1-\frac12 t^{-1})^t  \geq c_1 t^{\frac{-\alpha+1}{\beta}}. \qedhere
\end{align}
\end{proof}

\subsubsection{Proof of Theorem \ref{lemma:OU}}

We continue with the proof of Theorem \ref{lemma:OU} and study random walks on an Ornstein-Uhlenbeck process.

\begin{proof}[Proof of Theorem \ref{lemma:OU}]
Fix $n$, a sequence $t_k$ and $a\in \{-1,1\}^n$. Define the additional events $A_1=\{\sign(\xi_{-t_k})=1, 1\leq k\leq n\}$ and $\overline{A}=\{\sign(\xi_{t})=1, t\leq 0\}$.

We will use the following sequence of stochastic domination:
\begin{align}\label{eq:stochastic-dom}
\bP(\xi_0 \in \cdot \;|\; A) \preccurlyeq \bP(\xi_0 \in \cdot \;|\; A_1) \preccurlyeq \bP(\xi_0 \in \cdot \;|\; \overline{A}).
\end{align}
Here $\bP(\xi_0 \in \cdot \;|\; \overline{A})$ is the limit of $ \bP(\xi_0 \in \cdot \;|\;\sign(\xi_{s})=1, -T\leq s\leq 0)$ as $T\to\infty$, which exists and has Lebesgue-density $x\exp(-\frac12 x^2)$ on $[0,\infty)$ (see \cite{MandlSpectralTheory1961}).
The argument for the stochastic domination in \eqref{eq:stochastic-dom} is based on the following fact: Let $Y^1$ and $Y^2$ be two diffusions given by $dY^i_t=b^i_t(Y^i_t)dt+\sigma dW_t$. If $b_t^1\leq b_t^2$ and $\cL(Y^1_0)\preccurlyeq \cL(Y_0^2)$, then 
\begin{align}\label{eq:diffusion-domination}
\cL(Y^1_t)\preccurlyeq \cL(Y_t^2)\qquad \forall\ t\geq0.
\end{align}

To apply this to the first stochastic domination in \eqref{eq:stochastic-dom} holds, 
let $-t_l$ is the biggest time point with $a_l=-1$. Clearly $\bP(\xi_{-t_l}\in\cdot|A) \preccurlyeq \bP(\xi_{-t_l}\in\cdot|A_1)$. Furthermore, after $-t_l$ the events $A$ and $A_1$ agree past $-t_l$, that means that after $t_l$ we condition on the same event. This conditioning changes the drift to some new and time-inhomogeneous drift, for which only the initial law varies, and by \eqref{eq:diffusion-domination} we obtain the stochastic domination.

For the second stochastic domination, we use \eqref{eq:diffusion-domination} and the fact that conditioning the Ornstein-Uhlenbeck process on $\overline{A}$ further increases the drift compared to condition on $A_1$ (with the convention that the drift is $+\infty$ for $x\leq0$ when conditioning on $\overline{A}$).

An analogous bound to \eqref{eq:stochastic-dom} holds in the other direction when we condition the process to be negative, and $\bP(\xi_0\in\cdot|\underline{A})=\bP(-\xi_0\in\cdot|\overline{A})$. Together this implies
\begin{align}\label{eq:OU-1}
\bP(\abs{\xi_0} \in \cdot \;|\; A) \preccurlyeq \bP(\xi_0 \in \cdot \;|\; \overline{A}).
\end{align}

A bound on the total variation is then given by a coupling:
\begin{align}
\norm{\bP(\xi_t \in \cdot \;|\; A) - \bP(\xi_t\in \cdot)}_{TV} \leq \int \hP_{x,y}(\tau>t) \pi_A(dx,dy),
\end{align}
where $\hP_{x,y}$ is a coupling of two OU-processes $\xi^1_t$ and $\xi^2_t$ starting in $x$ and $y$ and $\pi_A$ is any coupling of $\bP(\xi_0 \in \cdot \;|\; A)$ with a normal distribution, and $\tau$ is the coupling time.

We take $\hP_{x,y}$ to be the coupling where the driving Brownian motions are perfectly negatively correlated until the processes are coupled. Then the difference $D_t$ is an OU-process satisfying
\begin{align}
dD_t = -D_tdt + 2dW_t \quad \text{and}\quad D_0=x-y.
\end{align}
The coupling time $\tau$ is $\tau_0$, the first hitting time of 0 of $D_t$. Note that the coupling time increases if $\abs{x-y}$ increases, in particular when we replace $\abs{x-y}$ by $\abs{x}+\abs{y}$. With this fact, choosing $\pi_A$ to be the independent coupling, and \eqref{eq:OU-1} we get
\begin{align}
\int \hP_{x,y}(\tau>t) \pi_A(dx,dy) \leq \int_0^\infty \int_0^\infty \bP_{x+y}(\tau_0>t) xe^{-\frac{x^2}{2}} \frac{2}{\sqrt{2\pi}} e^{-\frac{y^2}{2}} dx dy.
\end{align}
To conclude the proof we use the fact that that $\bP_{x+y}(\tau_0>r+\log(x+y))$ is exponentially small in $r$.

The claim that this example satisfies the conditions of Theorem \ref{def assump2} is now a simple computation by telescoping over all sites in $B$ and using the fact that the last time a site $x\in\bZ^d$ could be observed is $-|x|/R$, where $R$ is the interaction range of the jump kernel $\alpha$.
\end{proof}

\subsubsection{Proof of Theorem \ref{thm EP Contact Process} }\label{sec proofs contact}

In this subsection we present the proof of Theorem \ref{thm EP Contact Process}. 
Before doing so, we first introduce some definitions and prove a general theorem, Theorem \ref{thm lln FKG}, from which Theorem \ref{thm EP Contact Process} follows.

Let $E=\{0,1\}$ and associate to the space $\Omega$ the partial ordering such that $\xi \leq \eta$  if and only if $\xi(x) \leq \eta(x)$ for all $x\in \mathbb{Z}^{d+1}$. An event $B \in \mathcal{F}$ is said to be \emph{increasing} if $\xi\leq\eta$ implies $1_{B}(\xi)\leq 1_{B}(\eta)$. If $\xi\leq\eta$ implies $1_{B}(\xi)\geq 1_{B}(\eta)$ then $B$ is called \emph{decreasing}.
For $\bP, \bQ \in \mathcal{M}_1(\Omega)$, we say that $\bP$ \emph{stochastically dominates} $\bQ$ if $\bQ(B) \leq \bP(B)$ for all $B\in \mathcal{F}$ increasing.
Furthermore, a measure $\bP \in \mathcal{M}_1(\Omega)$ is \emph{positively associated} if it satisfies $\bP(B_1 \cap B_2) \geq \bP(B_1) \bP(B_2)$ for any two increasing events $B_1,B_2 \in \mathcal{F}$. 
Following \cite{LiggettSteifSD2006}, we say that $\bP$ is \emph{downward FKG} if,  for every finite $\Lambda \subset \mathbb{Z}^{d+1}$, the measure $\bP(\cdot \mid \eta \equiv 0 \text{ on } \Lambda )$ is positively associated. 

\begin{thm}\label{thm lln FKG}
Let $\bP \in \mathcal{M}_1(\Omega)$ be downward FKG and assume that there exists $\phi \colon \nat \rightarrow [0,1]$ such that for all $ (x,s) \in \Lambda(l)$ and all $\gamma \in \bigcup_{k \geq 1} \Gamma_k$,
\begin{align}\label{eq supercritical}
\bP \left(\eta_s(x) =1 \mid \eta \equiv 0 \text{ along } \gamma \right) \geq \bP \left(\eta_0(o)=1 \right) - \phi(l), 
\\ \label{eq assump positive}
\bP( \eta_s(x) =1 \mid \eta \equiv 1 \text{ along } \gamma ) \leq \bP(\eta_0(o) =1) + \phi(l). \end{align}
If $\sum_{l \geq 1} l^d \phi(l) < \infty$, then the conditions of Theorem \ref{def assump2} are satisfied.
\end{thm}

\begin{rem}
In the above theorem, and throughout this section, we write ``$\eta \equiv i$ along $\gamma$'', where $i\in \{0,1\}$ and $\gamma \in \Gamma := \bigcup_{k \geq 1} \Gamma_k$, for the event that $\{ \eta_s(x)=i \: \forall \: (x,s) \in \gamma +[-R,R]^d\times\{0\} \}$. 
\end{rem}

\begin{proof}[Proof of Theorem \ref{thm lln FKG}]
Let $B \in \mathcal{F}$. For any $k \in \nat$, we have similar to the proof of Lemma \ref{lem exp main}  that
\begin{align}
&\left| P_{\bP}\left(\eta_k^{EP} \in B\right) - \bP\left(B\right) \right|
\\ \leq & \sum_{\gamma  \in \Gamma_k} \sum_{A_{-k}^{-1} \in \mathcal{A}_{-k}^{-1} (\gamma)}\bigg[ \left| \bP\left(B \mid A_{-k}^{-1} \right)  - \bP\left(B \mid \eta \equiv 0 \text{ along } \gamma \right)\right| 
\\ + & \left| \bP\left(B \right)  - \bP\left(B \mid \eta \equiv 0 \text{ along } \gamma \right)\right| \bigg] \bP^{-k} \left(X_{-k,\dots,0} = \gamma, A_{-k}^{-1} \right).
\end{align}
We next show that, under \eqref{eq supercritical} and \eqref{eq assump positive},  
\begin{align}
\sup_{B\in \Lambda(l)} \left| \bP\left(B \mid A_{-k}^{-1} \right)  - \bP\left(B \mid \eta \equiv 0 \text{ along } \gamma \right)\right| \rightarrow 0,\quad  \text{as } l \rightarrow \infty.
\end{align}

Fix $\gamma \in \Gamma_k$ and $A_{-k}^{-1} \in \mathcal{A}_{-k}^{-1}(\gamma)$. Since $\bP$ is downward FKG, it is the case that $\bP(\cdot \mid \eta \equiv 0 \text{ along } \gamma)$ is stochastically dominated by $\bP( \cdot \mid A_{-k}^{-1})$. Hence, by Strassens Theorem, there exists a coupling $\hP_{0,1}$ of $\bP(\cdot \mid \eta \equiv 0 \text{ along } \gamma)$ and $\bP( \cdot \mid A_{-k}^{-1})$ such that 
$\hP_{0,1} \left(\eta^1 \leq \eta^2 \right) =1$.
We moreover have that, for all $B \in \mathcal{F}_{\Lambda(l)}$, $l \in \nat$, 
\begin{align}
& |\bP(B \mid \eta \equiv 0 \text{ along } \gamma) -\bP( B \mid A_{-k}^{-1} )|  
\\ \leq &\hP_{0,1} \left( \eta^1 \neq \eta^2 \text{ on } \Lambda(l) \right)
 \\ \leq &\sum_{ (x,s) \in \Lambda(l)} \hP_{0,1} \left( \eta^1_s(x) \neq \eta^2_s(x) \right) 
\\ =  &\sum_{ (x,s) \in \Lambda(l)} \hP_{0,1} \left( \eta^1_s(x)=0, \eta^2_s(x)=1 \right) 
\\  =  &\sum_{ (x,s) \in \Lambda(l)} \left( \hP_{0,1} ( \eta^1_s(x)=0) - \hP_{0,1} ( \eta^2_s(x)=0) \right).
\end{align}
Furthermore, since $\bP$ is downward FKG,  we know that 
\begin{align}
&\hP_{0,1} \left( \eta^1_s(x)=0 \right) - \hP_{0,1} \left( \eta^2_s(x)=0\right)
  \\ = &\bP \left( \eta_s(x) = 0 \mid \eta \equiv 0 \text{ along } \gamma \right) - \bP\left(\eta_s(x) =0 \mid A_{-k}^{-1} \right) 
 \\ \leq &\bP \left( \eta_s(x) = 0 \mid \eta \equiv 0 \text{ along } \gamma \right) - \bP\left(\eta_s(x) =0 \mid \eta \equiv 1 \text{ along } \gamma \right) 
\end{align}
As a consequence, by using \eqref{eq supercritical} and \eqref{eq assump positive}, we obtain by the derivations above that
\begin{align}\label{eq dfkg help100}
\sup_{B \in \mathcal{F}_{\Lambda(l)} } |\bP(B \mid \eta \equiv 0 \text{ along } \gamma) -\bP( B \mid A_{-k}^{-1} )|  \leq C \sum_{t\geq l} t^d \phi(t),
\end{align}
 for some constant $C\in (0,\infty)$. By a word by word adaptation of this argument, replacing $\bP\left(B \mid A_{-k}^{-1} \right)$ by $\bP\left(B \right)$, it can similarly be shown that
\begin{align}\label{eq dfkg help101}
\sup_{B\in \Lambda(l)}  \left| \bP\left(B \right)  - \bP\left(B \mid \eta \equiv 0 \text{ along } \gamma \right) \right| \leq C\sum_{t\geq l} t^d \phi(t).
\end{align}
 Substituting the estimates from \eqref{eq dfkg help100} and \eqref{eq dfkg help101} into the first inequality of this proof, and using that $\lim_{l \rightarrow \infty} \sum_{t\geq l} t^d \phi(t) =0$, we obtain  that the conditions of Theorem  \ref{def assump2} are satisfied. 
\end{proof}

We continue with the proof of Theorem \ref{thm EP Contact Process}.

\begin{proof}[Proof of Theorem \ref{thm EP Contact Process}]
Let $(\xi_t)$ be the contact process on $\mathbb{Z}^{\tilde{d}}$ with $\tilde{d}\geq1$ and $\lambda> \lambda_c(\tilde{d})$. This process is known to satisfy the downward FKG property, as shown by \cite{BergHaggstromKahn2006}, Theorem 3.3 (see also Lemma 2.1 in \cite{BergBethuelsenCP2016}). Thus, for the proof of Theorem \ref{thm EP Contact Process}, it is sufficient to show that  \eqref{eq supercritical}  and \eqref{eq assump positive} holds. In fact, it is sufficient to show  that the estimates of Theorem \ref{thm lln FKG} hold for sites $(o,s)$ with $s\in \integers_{\geq0}$. To see this, recall the graphical representation of the contact process (see p.\ 32-34 in \cite{LiggettSIS1999}). Since the spread of information is bounded by a Poisson process with rate $2d\lambda$, it is evident that the finite speed of propagation property holds, and thus that Corollary \ref{thm:HMM-2} applies.

That \eqref{eq assump positive} holds for the contact process is now a simple application of the graphical representation and the fact that the contact process started from all sites equal to $1$ converges exponentially fast towards the upper invariant measure. See \cite{LiggettSIS1999}, Theorem 1.2.30, and the remark directly after for estimates of the latter. In particular, \eqref{eq assump positive}  holds with $\phi(l)$ exponentially decaying in $l$. Note that, this estimate holds for $(\xi_t)$, that is, without applying the projection map.

In order to conclude a similar estimate for \eqref{eq supercritical}, on the other hand, we restrict to the projection of $(\xi_t)$ onto the one dimensional lattice. In this case, \eqref{eq supercritical}, again with $\phi(l)$ exponentially decaying in $l$, is a direct application of  \cite{BergBethuelsenCP2016}, Theorem 1.7. Thus, by Theorem \ref{thm lln FKG}, we conclude that the conditions of Theorem \ref{def assump2} are satisfied.
\end{proof}

\begin{rem}
The statement of Theorem \ref{thm EP Contact Process} can be extended to projection maps from $\mathbb{Z}^{\tilde{d}}$ to $\mathbb{Z}_{\tilde{d}-1}^{\tilde{d}}:= \integers^{\tilde{d}-1} \times \{0\}$ for any $\tilde{d}\geq 2$ and $\lambda>\lambda_c(\tilde{d})$. Indeed, Theorem 1.7 in \cite{BergBethuelsenCP2016} still holds in this generality.
\end{rem}

\subsection{Proof of Theorem \ref{thm continuity} and Theorem \ref{thm continuity2}}

\begin{proof}[Proof of Theorem \ref{thm continuity}]
We first show continuity with respect to $\left(\bP^{EP(n)}\right)$. Let $\epsilon>0$, and let $m \leq n$ with $n,m \in \nat$. For $\Lambda \subset \Lat$ finite and $B\in \mathcal{F}_{\Lambda}$ we have that, for every $t \in \nat$,
\begin{align}
|\bP^{EP(m)}(B) - \bP^{EP(n)}(B)| &\leq  |\bP^{EP(m)}(B) - P_{\bP_m}(\eta_t^{EP(m)}\in B)|
\\ &+  |\bP^{EP(n)}(B) - P_{\bP_n}(\eta_t^{EP(n)}\in B)|
\\ &+ | P_{\bP_n}(\eta_t^{EP(n)}\in B) - P_{\bP_m}(\eta_t^{EP(m)}\in B)|.
\end{align}
By Assumption c) we can fix $t$ such that the sum of the first two terms is less than $\epsilon/2$. By the uniformity assumption, this bound holds irrespectively of $m$ and $n$. It thus remains to show that also the third term can be made smaller than $\epsilon/2$ by possibly taking $m$ large.
To this end, we use the expansion in \eqref{eq expansion 2.1}, and note that 
\begin{align}
&P_{\bP_m}(\eta_t^{EP(m)}\in B) 
\\ = &\sum_{\substack{ \gamma \in \Gamma_t \\ A_{-t}^{-1} \in \mathcal{A}_{-t}^{-1}(\gamma) }} \bP_m \left( B,  A_{-t}^{-1} \right) P_m \left(X_{-t,\dots,0}=\gamma \mid A_{-t}^{-1} \right)
\\ = &\sum_{\substack{ \gamma \in \Gamma_t \\ A_{-t}^{-1} \in \mathcal{A}_{-t}^{-1}(\gamma)} } \left( \bP_n ( B,  A_{-t}^{-1} ) \pm \delta_{1,m}(t) \right) \left(P_n (X_{-t,\dots,0}=\gamma \mid A_{-t}^{-1} ) \pm \delta_{2,m}(t) \right),
\end{align}
where, due to a) and b), the error terms $\delta_{1,m}(t)$ and $ \delta_{2,m}(t) $ approaches $0$ as $m \rightarrow \infty$. 
In particular, again since
\begin{align}
P_{\bP_n}(\eta_t^{EP(n)}\in B) = \sum_{\substack{ \gamma \in \Gamma_t \\ A_{-t}^{-1} \in \mathcal{A}_{-t}^{-1}(\gamma) }} \bP_n \left( B,  A_{-t}^{-1} \right) P_n \left(X_{-t,\dots,0}=\gamma \mid A_{-t}^{-1} \right),
\end{align}
by taking $m$ large enough we can guarantee that
\begin{align}
| P_{\bP_n}(\eta_t^{EP(n)}\in B) - P_{\bP_m}(\eta_t^{EP(n)}\in B )|<\epsilon/2.
\end{align} Since this bound holds for all $n\geq m$ it follows that $(\bP^{EP(m)}(B))$ is a Cauchy sequence and hence converges to a limit. Moreover, since $B$ and $\Lambda$ were arbitrary, this is true for any local local event $B \in \mathcal{F}$. This implies that $\bP^{EP(m)}$ converges weakly to $\bP^{EP}$ for some $\bP^{EP} \in \mathcal{M}_1(\Omega)$.

We next proceed with the proof of $P_{\bP}(\eta_t^{EP} \cdot ) \implies \bP^{EP}$, where $\bP^{EP}$ is the limiting measure above. 
Let $\epsilon>0$ and $B\in \mathcal{F}$ local. For any $n \in \nat$, we have that
\begin{align}
|\bP^{EP}(B) - P_{\bP}(\eta_t^{EP} \in B)| &\leq |\bP^{EP}(B) - \bP^{EP(n)}(B)|
\\ &+ |\bP^{EP(n)}(B) - P_{\bP_n}(\eta_t^{EP(n)}\in B)|
\\ &+|P_{\bP_n}(\eta_t^{EP(n)}\in B) - P_{\bP}(\eta_t^{EP} \in B) |. 
\end{align}
Fix $t $ such that the second term is smaller than $\epsilon/3$. This we can do by applying Assumption c). Next, by taking $n$ large the first term can be made smaller then $\epsilon/3$ as well since $\bP^{EP(n)} \implies \bP^{EP}$, as we have shown above. For the third term we can proceed as in for the proof of $\bP^{EP(n)} \implies \bP^{EP}$ above. Indeed, since $t$ is fixed, we can use that $\bP_n \implies \bP$ and that  $\epsilon(n) \downarrow 0$ together with the finite range assumption of the random walk. Hence we may take $n$ so large that also the third term is less that $\epsilon/3$. Since $\epsilon>0$ was taken arbitrary, this shows that $P_{\bP}(\eta_t^{EP} \in B) \rightarrow \bP^{EP}(B)$ as $t \rightarrow \infty$. Since $B\in \mathcal{F}$ was an arbitrary local event, we conclude that $P_{\bP}(\eta_t^{EP} \in \cdot)$ converges weakly towards $\bP^{EP}(\cdot)$. As a necessary consequence, it also follows that by standard arguments that  $\bP^{EP}$ is invariant with respect to $(\eta_t^{EP})$.

\end{proof}

\begin{proof}[Proof of Theorem \ref{thm continuity2}]
Let $\bP_{\sigma}$ be the path measure of $(\eta_t)$ when started from $\sigma \in \Omega_0$ and assume that  $(\eta_t^{EP})$  is uniquely ergodic with invariant measure $\mu^{EP} \in \mathcal{M}_1(\Omega)$. 
We have that, for any $B \in \mathcal{F}_{\Lambda}$, $\Lambda \subset \mathbb{Z}^{d} \times \{0\}$ finite, and any $t\in \nat$,
\begin{equation}
\begin{aligned}\label{eq uniq ergodic 2}
\left| \mu^{EP}(B) - \mu^{EP(n)}(B)\right|
 = &\left| t^{-1} \sum_{k=1}^t  \left[P_{\mu^{EP}}(\eta_k^{EP}\in B) - P_{\mu^{EP(n)}}(\eta_k^{EP(n)} \in B) \right]\right|
\\ \leq &\left| t^{-1} \sum_{k=1}^t  \left[P_{\mu^{EP}}(\eta_k^{EP}\in B) - P_{\mu^{EP(n)}}(\eta_k^{EP(n)} \in B) \right] \right|
\\ + &\left| t^{-1} \sum_{k=1}^t  \left[P_{\mu^{EP(n)}}(\eta_k^{EP}\in B) - P_{\mu^{EP(n)}}(\eta_k^{EP(n)} \in B) \right] \right|.
\end{aligned}
\end{equation}
Since $(\eta_t^{EP})$ is uniquely ergodic, 
it follows by classical arguments that
\begin{align}
\sup_{\sigma,\omega \in \Omega_0} \left| t^{-1} \sum_{k=1}^t  \left[ P_{\sigma}(\eta_k^{EP}\in B) - P_{\omega}(\eta_k^{EP} \in B) \right] \right| \rightarrow 0
 \end{align}
as $t$ approaches $\infty$  (see e.g.\ Theorem 4.10 in \cite{EinsiedlerWardErgodic2011}). Hence, by taking $t$ large we can assure that the first term of the r.h.s.\ of \eqref{eq uniq ergodic 2} is less than $\epsilon/2$. Next, for the second term, we have that, for any fixed $t>0$,
\begin{align} 
t^{-1} \sum_{k=1}^t  \left| P_{\mu^{EP(n)}}(\eta_k^{EP}\in B) - P_{\mu^{EP(n)}}(\eta_k^{EP(n)} \in B)\right|\rightarrow 0 \text{ as } n \rightarrow \infty.
\end{align}
This follows similarly as in the proof of Theorem \ref{thm continuity}. Indeed, for each $k \in \{1, \dots,t\}$, we have that
\begin{align}
&P_{\mu^{EP(n)}}(\eta_k^{EP}\in B) 
\\ = &\sum_{\substack{ \gamma \in \Gamma_k \\ A_{-k}^{-1} \in \mathcal{A}_{-k}^{-1}(\gamma) }} \bP_{\mu^{EP(n)}} \left( B,  A_{-k}^{-1} \right) P_n \left(X_{-k,\dots,0}=\gamma \mid A_{-k}^{-1} \right)
\\ = &\sum_{\substack{\gamma \in \Gamma_k \\ A_{-k}^{-1} \in \mathcal{A}_{-k}^{-1}(\gamma) }} \left( \bP_{\mu^{EP}} ( B,  A_{-t}^{-1} ) \pm \delta_{1,n}(t) \right)
\\ &\quad \quad \quad \quad \quad \quad \quad \cdot \left(P (X_{-k,\dots,0}=\gamma \mid A_{-k}^{-1})  \pm\delta_{2,n}(t) \right),
\end{align}
where both the error terms $\delta_{1,n}(t)$ and $ \delta_{2,n}(t) $ approaches $0$ as $n \rightarrow \infty$. Thus, by taking $n$ sufficiently large we can assure that the second term on the r.h.s. of \eqref{eq uniq ergodic 2} is less than $\epsilon/2$. From this we conclude that $\left| \mu^{EP}(B) - \mu^{EP(n)}(B)\right|<\epsilon$ for all $n$ large. Since $B$ and $\Lambda$ were arbitrary chosen, we hence conclude the proof.
\end{proof}

 \subsection{Proof of Theorem \ref{thm alternative EP}}
 \begin{proof}[Proof of Theorem \ref{thm alternative EP}]

The main part of the proof goes along the same lines as the proof of Theorem \ref{def assump1}. The main difference is an estimate which is similar to Lemma \ref{lem exp main} and which we present next. Let $B \in \mathcal{F}_{\geq0}$. We have that, for any $k \in \nat$,
\begin{align}
&\big| P_{\bP}\left(\eta_k^{EP} \in B\right) - \bP\left(B\right) \big|
\\= &\big| \sum_{\gamma  \in \Gamma_k} \sum_{A_{-k}^{-1} \in \mathcal{A}_{-k}^{-1} (\gamma)} \bP\left(B, A_{-k}^{-1} \right) P \left(X_{-k,\dots,0} = \gamma\mid A_{-k}^{-1} \right) - \bP\left(B\right)\big|
 \\ \leq & \sum_{\gamma  \in \Gamma_k} \sum_{A_{-k}^{-1} \in \mathcal{A}_{-k}^{-1} (\gamma)} \big| \bP\left(B \mid A_{-k}^{-1} \right) - \bP\left(B\right)\big| \bP^{-k} \left(X_{-k,\dots,0} = \gamma, A_{-k}^{-1} \right), 
\end{align}
where in  the last equality we used the fact that,
\begin{align}
\sum_{\gamma  \in \Gamma_k}  \sum_{A_{-k}^{-1} \in \mathcal{A}_{-k}^{-1} (\gamma)}  \bP^{-k} \left(X_{-k,\dots,0} = \gamma, A_{-k}^{-1}(\gamma,\sigma) \right) =1.
\end{align}
Consequently, by the bound in \eqref{eq thm alternative EP}, we conclude that,  for any $k \in \nat$,
\begin{align}\label{eq thm alternative help}
\big| P_{\bP}\left(\eta_k^{EP} \in B\right) - \bP\left(B\right) \big| \leq M_1 \bP(B), \quad \forall \: B \in \mathcal{F}_{\geq0}.
\end{align}
In particular, $P_{\bP}\left(\eta_k^{EP} \in \cdot \right) \ll \bP$ on $\mathcal{F}_{\geq0}$ and $\frac{ dP_{\bP}\left(\eta_k^{EP} \in\cdot\right)|_{\mathcal{F}_{\geq0}}}{d\bP(\cdot)|_{\mathcal{F}_{\geq0}}} \leq M_1$.

Let $\bQ \in \mathcal{M}_1(\Omega)$ be a limiting measure of the sequence $\left(t^{-1} \sum_{k=1}^tP_{\bP}(\eta_k^{EP} \in \cdot)\right)_{t >0}$, by possibly taking sub-sequential limits. Then, by means of weak convergence,  since the space of $M_1$-bounded functions on a compact space form a compact space, and the limit of  bounded measurable functions is measurable, \eqref{eq thm alternative help} immediately transfers to $\bQ$. Consequently, we have $\bQ \ll \bP$ on $\mathcal{F}_{\geq0}$ and $\frac{ d\bQ|_{\mathcal{F}_{\geq0}}}{d\bP|_{\mathcal{F}_{\geq0}}} \leq M_1$. This concludes the first part.

Next, assume that \eqref{eq thm alternative EP2} holds from which it follows that, for every $B \in \mathcal{F}_{\geq0}$,
\begin{align}
\left| \bP(B) - \bP(B \mid A_k) \right| \leq M_2 \bP(B\mid A_k), \quad \forall \: A_k \in \mathcal{A}_{-kl}^{-1}, \: k \in \nat,
\end{align}
Similarly to how we obtained \eqref{eq thm alternative help}, we hence conclude that, for any $k \in \nat$,
\begin{align}\label{eq thm alternative help24}
\big| P_{\bP}\left(\eta_k^{EP} \in B\right) - \bP\left(B\right) \big| \leq M_2 P_{\bP}\left(\eta_k^{EP} \in B\right), \quad \forall \: B \in \mathcal{F}_{\geq0}.
\end{align}
From this estimate, and using the same argument as for the proof of the first part, we hence conclude that $\bP \ll \bQ$ and that $\frac{ d\bP|_{\mathcal{F}_{\geq0}}}{d\bQ|_{\mathcal{F}_{\geq0}}} \leq M_2$.
\end{proof}

\subsection{Strong disagreement percolation}\label{sec SDP}

\subsubsection{Basic disagreement percolation}\label{subsection:basicdisagreementpercolation}
For simplicity we assume that $E=\{0,1\}$ and that the environment $(\eta_t)$ is a translation invariant nearest neighbour probabilistic cellular automaton (PCA). Further, let $c_i(\eta):=\bP_\eta(\eta_1(o)=i), i=0,1$. By the nearest neighbour property, $c_i(\eta)=c_i(\xi)$ if $\eta(x)=\xi(x)$ for all $\abs{x}\leq1$.

The evolution of the PCA can be constructed by a sequence $(U_t(x))_{x\in\bZ^d, t \geq 1}$ of i.i.d. $[0,1]$-uniform variables in an iterative way: given $\eta_t$, $\eta_{t+1}(x):=\ind_{U_{t+1}(x)\leq c_1(\theta_{x} \eta_t)}$, $x\in\bZ^d$. Here $\theta_{x}$ is the shift on $\mathbb{Z}^{d}$, that is, for $\eta \in \Omega_0$ we have $(\theta_{x} \eta)(y) = \eta(y+x)$, $y \in \mathbb{Z}^{d}$

This construction allows for coupling of $\bP_{\eta^1}$ and $\bP_{\eta^2}$, the graphical construction coupling, by using the same set of $[0,1]$-uniform i.i.d. variables $(U_t(x))$. The starting point of disagreement percolation is the observation that the value of $\eta_{t+1}(x)$ is sometimes independent of $\eta_t$, namely if either $U_{t+1}(x)< c_-:=\inf_{\eta\in\Omega}c_1(\eta)$ or $U_{t+1}(x)>c_+:=\sup_{\eta\in\Omega}c_1(\eta)$. This allows the environment to forget information, which can be encoded in the coupling.
 The disagreement percolation is then the triple $(\eta^1_t,\eta^2_t,\xi_t)_{t\geq 0}$, where $\eta^i_t$ is constructed from the initial configuration $\eta^i$ and the $(U_t(x))$, and $\xi_t$ is given by $\xi_0(x)=\ind_{\eta^1(x)\neq \eta^2(x)}$ and for $t>0$;
\begin{align}\label{eq:DP}
 \xi_t(x) = \begin{cases}
 1,\quad & U_t(x)\in [c_-,c_+] \text{ and } \exists\;y, \abs{y-x}\leq 1 : \xi_{t-1}(y)=1;\\
 0, &otherwise.
 \end{cases}
\end{align}
The name disagreement percolation comes from the fact that $\xi_t(x)=0$ implies $\eta^1_t(x)=\eta^2_t(x)$ and $(\xi_t)$ is a directed site percolation process with percolation parameter $p=c_+-c_-$. We denote the law of this so constructed triple $(\eta^1_t,\eta^2_t,\xi_t)_{t\geq 0}$  by $\hP_{\eta^1,\eta^2}$.

\begin{defn}
If $p=c_+-c_-<p_c$, where $p_c$ is the critical value of directed site percolation in $\bZ^d$, then we say that the disagreement percolation $\hP$ is subcritical.
\end{defn}

\begin{rem}
This coupling can be improved by looking at more information. For example the site percolation model does not use the total number of neighbours which satisfy $\xi_{t-1}(y)=1$, only that the indicator that this number is positive. By taking this information into account when deciding based on whether $\xi_t(x)$ should be 1 or 0 the range of PCA where the disagreement percolation is subcritical can be extended. 
\end{rem}

\begin{rem}
If the disagreement percolation coupling is subcritical, then necessarily there is a uniquely ergodic measure for the process $(\eta_t)$, as follows by standard coupling arguments and comparison with subcritical directed site percolation.
\end{rem}

\subsubsection{Disagreement percolation and backward cones}\label{subsection:conedisagreementpercolation}
The disagreement percolation we introduced in the previous subsection is a way to control the influence of the initial configuration on the future, by giving an upper bound on the space-time points which depend on differences in the initial configurations. In the context of this article we want something slightly different, namely to control the influence of a backwards cone. With this in mind we construct a different version of the disagreement percolation coupling.

Denote by $\bP^{-\infty}_\mu$ the law of $(\eta_t)_{t\in\bZ}$ under the stationary law $\mu$ and by $(U_t(x))_{t\in\bZ,x\in\bZ^d}$ the i.i.d. uniform $[0,1]$ variables of the corresponding graphical construction. Denote by $C_{b}:=\{(x,t) \in \bZ^d\times\{...,-1,0\} : \abs{x}\leq\abs{t} \}$ the infinite backward cone with tip at $(0,0)$ and by  $\cC_b:=\sigma(\eta_t(x) : (x,t)\in C_b)=\sigma(U_t(x) : (x,t)\in C_b)$ the $\sigma$-algebra generated by the sites which lie in the cone $C_b$.

Let $A, B\in\cC_b$. We now construct the disagreement percolation process $(\eta^1_t,\eta^2_t,\xi_t)_{t\in\bZ}$ with law $\hP_{A,B}$, where $\eta^1$ has law $\bP^{-\infty}_\mu(\cdot|A)$ and $\eta^2$ has law $\bP^{-\infty}_\mu(\cdot|B)$. The idea is almost the same as in Subsection \ref{subsection:basicdisagreementpercolation}, the only difference is on the cone $C_b$. On $C_b$, we draw $(\eta^1_t(x))_{(x,t)\in C_b}$ from $\bP^{-\infty}_\mu(\cdot|A)$, independently $(\eta^2_t(x))_{(x,t)\in C_b}$ from $\bP^{-\infty}_\mu(\cdot|B)$, and set $\xi_t(x)=1$ for $(x,t)\in C_b$. Outside $C_b$, $(\eta_t^1,\eta^2,\xi)$ evolves like the basic disagreement percolation coupling by using the same $(U_t(x))_{(x,t)\in C_b^c}$. As the evolution outside $C_b$ is the same as the basic disagreement percolation, the definition of subcriticality remains the same.

\begin{lem}
Suppose the disagreement percolation is subcritical. Then the environment satisfies \eqref{eq bcm}.
\end{lem}

\begin{proof}
Let $A\in \cC_b$ be arbitrary and let $\hP_A$ be the disagreement percolation coupling of $\bP^{-\infty}_\mu(\cdot|A)$ and $\bP^{-\infty}_\mu$. We then have for any $B \in \cF_{\Lambda(l)} $, $l \geq 1$, 
\begin{align}
\abs{\bP^{-\infty}_\mu(B|A) - \bP^{-\infty}_\mu(B)} & \leq \hP_A(\exists (x,t) \in \Lambda(l) \colon \xi_t(x)=1 ),
\end{align}
which is exponentially small in $l$ and independent of the choice of $B$ and $A$.
\end{proof}

\subsubsection{Strong disagreement percolation}
We say that $(\eta^1_t,\eta^2_t,\xi_t)$ is a strong disagreement percolation coupling if $\xi_t(x)=0$ implies $\eta^1_t(x)=\eta^2_t(x)$ and $\eta^1$ and $\xi$ are independent. This independence is a stronger assumption than regular disagreement percolation.
\begin{lem}\label{lem sdp subcritical}
Suppose $p^*:=\max((c_+-c_-)/c_+, (c_+-c_-)/(1-c_-))<p_c$. Then there exist strong versions of the disagreement percolation couplings in Sections \ref{subsection:basicdisagreementpercolation} and \ref{subsection:conedisagreementpercolation}.
\end{lem}
\begin{proof}
The basic concept of the construction is similar to the regular disagreement percolation. The difference is that we no longer use a single $U_t(x)$ to build the processes $(\eta_t^1(x),\eta_t^2(x),\xi_t(x))$ from $(\eta_{t-1}^1,\eta_{t-1}^2,\xi_{t-1})$. Instead we take three $[0,1]$-uniform i.i.d random variables $(U_t(x)^1,U_t(x)^2,U_t(x)^3)$. We then set 
\mathtoolsset{showonlyrefs=false}
\begin{align}
\begin{aligned}\label{eq:sdp}
\eta^1_t(x)&=\ind_{U^1_t(x)\leq c_1(\theta_{-x}\eta^1_{t-1})};	\\
\xi_t(x) &=\ind_{U^3_t(x)\leq p^*} \ind_{\exists y,\abs{y-x}\leq 1: \xi_{t-1}(y)=1};	\\
\eta^2_t(x) &=\begin{cases}
\eta^1_t(x), \quad &\xi_t(x)=0;\\
1, &U^2_t(x)\leq \frac{c_1(\theta_{-x}\eta^2_{t-1})-(1-p^*)c_1(\theta_{-x}\eta^1_{t-1})}{p^*}\text{ and }\xi_t(x)=1;\\
0,&otherwise.
\end{cases}
\end{aligned}
\end{align}
\mathtoolsset{showonlyrefs=true}
The choice of $p^*$ guarantees that $\frac{c_1(\theta_{-x}\eta^2_{t-1})-(1-p^*)c_1(\theta_{-x}\eta^1_{t-1})}{p^*}\in [0,1]$, and a direct computation shows that the probability that $\eta^2_t(x)=1$ is $c_1(\theta_{-x}\eta^2_{t-1})$.
\end{proof}

\subsubsection{Proof of Theorem \ref{thm:sdp} and Corollary \ref{cor SDP}}

\begin{proof}[Proof of Theorem \ref{thm:sdp}]
The proof is based on a coupling argument. Let 
\begin{align}
C_{-k}:=\{(x,t)\in\bZ^d\times\{...,-k-1,-k\}: \abs{x-\gamma_{-k}}\leq \abs{t-k} \}
\end{align}
be the infinite backwards cone with tip at $(\gamma_{-k},-k)$. We construct iteratively the random variables $(\eta_t(x)^{1,m}, \eta_t(x)^{2,m},\xi_t(x)^m) \in E\times E\times \{0,1\}, (x,t)\in C_{-k-1+m}$, and $H^m \in \bN$, and denote their law by $\tP^m$. We start with $\eta_t(x)^{1,0}$ and $\eta_t(x)^{2,0},\xi_t(x)^0$ chosen independently from $\bP^{-\infty}_\mu\left(\cdot \right)$ and $\bP^{-\infty}_\mu\left(\cdot \; \middle|\; A_{-k-1}^{-k-1} \right)$ restricted to the cone $C_{-k-1+m}$ and set $\xi_t^0(x)=1$ for $(x,t)\in C_{-k-1}$, and $H^0=0$.

Given $\tP^m$, let $\tP^{m,*}$ be the extension of $\tP^m$ to the cone $C_{-k+m}$ based on the strong disagreement percolation coupling, that is 
\begin{align}
(\eta^{1,m}_t(x),\eta^{2,m}_t(x),\xi^m_t(x))_{(x,t)\in C_{-k+m}\setminus C_{-k-1+m}}
\end{align}
 are distributed according to the evolution described in \eqref{eq:sdp}.

The general strategy is as follows: We want to condition the measure $\tP^{m,*}$ on the event $\{\eta^{1,m} \in A_{-k}^{-k+m}, \eta^{2,m} \in A_{-k-1}^{-k+m}\}$. Observe that, on $\xi_{-k+m}(\gamma_{-k+m})=0$, the events $\eta^{1,m} \in A_{-k+m}^{-k+m}$ and $\eta^{2,m} \in A_{-k+m}^{-k+m}$ are equivalent. This is the good case. The bad case is when $\xi_{-k}(\gamma_{-k})=1$. In this event, we simply restart the coupling procedure by coupling $\bP^{-\infty}_\mu\left(\cdot \; \middle|\; A_{-k-1}^{-k} \right)$ and $\bP^{-\infty}_\mu\left(\cdot \; \middle|\; A_{-k}^{-k} \right)$ independently. The role of $H^m$ is to keep track of the number of iterations since the last time we had to reset and try again. 

Define 
\begin{align}
q_{m}(h) := &\tP^{m,*}\left(\xi^{m}_{-k+m}(\gamma_{-k+m})=0 \;\middle|\; \eta^{1,m} \in A_{-k}^{-k+m}, H^m=h \right)\\
&\wedge \tP^{m,*}\left(\xi^{m}_{-k+m}(\gamma_{-k+m})=0 \;\middle|\; \eta^{2,m} \in A_{-k-1}^{-k+m}, H^m=h \right);\\
Q^1_m := &\tP^{m,*}\left(\xi^{m}_{-k+m}(\gamma_{-k+m})=0 \;\middle|\; \eta^{1,m} \in A_{-k}^{-k+m} \right) - \sum_{h\geq 0}q_m(h);\\
Q^2_m := &\tP^{m,*}\left(\xi^{m}_{-k+m}(\gamma_{-k+m})=0 \;\middle|\; \eta^{2,m} \in A_{-k-1}^{-k+m} \right) - \sum_{h\geq 0}q_m(h) .
\end{align}
Let $B^1, B^2, D \in \sigma(C_{-k+m})$. We now define $\tP^{m+1}$ based on $\tP^m$ by
\begin{align}
&\tP^{m+1}\left(\eta^{1,m+1}\in B^1, \eta^{2,m+1}\in B^2, \xi^{m+1}\in D, H^{m+1}=h+1\right)\\
&:=q_m(h)\tP^{m,*}\left( \eta^{1,m}\in B^1, \eta^{2,m}\in B^2, \xi^{m}\in D, H^{m}=h \;\mid \;  \right. \\
& \quad \qquad \qquad\qquad\left. \eta^{1,m}\in A_{-k+m}^{-k+m}, \eta^{2,m}\in A_{-k+m}^{-k+m}, \xi^m_{-k+m}(\gamma_{-k+m})=0\right)\end{align}
and
\begin{align}
&\tP^{m+1}\left(\eta^{1,m+1}\in B^1, \eta^{2,m+1}\in B^2, \xi^{m+1}\in D, H^{m+1}=0\right)\\
&:=\left[
\tP^{m,*}\left( \eta^{1,m}\in B^1 \;\mid\; \eta^1_{-k+m}\in A_{-k}^{-k+m}, \xi^m_{-k+m}(\gamma_{-k+m})=0\right) Q^1_m
  \right.   
\\&+ \left. \tP^{m,*}\left(\eta^{1,m}\in B^1, \xi^m_{-k+m}(\gamma_{-k+m})=1\;\middle|\; \eta^{1,m}\in A_{-k}^{-k+m}\right)  \right]  \frac{1}{1-\sum_{h\geq 0}q_m(h) }
\\&\cdot\left[\tP^{m,*}\left( \eta^{2,m}\in B^2 \;\middle|\; \eta^{1,m}_{-k+m}\in A_{-k}^{-k+m}, \xi^m_{-k+m}(\gamma_{-k+m})=0\right) Q^2_m \right. 
\\&+\left.  \tP^{m,*}\left(\eta^{2,m}\in B^2, \xi^m_{-k+m}(\gamma_{-k+m})=1\;\middle|\; \eta^{2,m}\in A_{-k}^{-k+m}\right)  \right]\ind_{1 \text{ on }C_{-k+m}\in D}.
\end{align}
A direct computation shows that $\tP^{m+1}(\eta^{1,m+1}\in B^1)=\bP^{-\infty}_\mu(B^1 \;|\; A_{-k}^{-k+m})$ and $\tP^{m+1}(\eta^{2,m+1}\in B^2)=\bP^{-\infty}_\mu(B^2 \;|\; A_{-k}^{-k+m})$
assuming that $\tP^{m}$ satisfies the corresponding properties. Therefore $\tP^{k}$ extended to all space-time points using the strong disagreement percolation construction \eqref{eq:sdp} is a coupling of $\bP^{-\infty}_\mu(\cdot|A_{-k}^{-1})$ and $\bP^{-\infty}_\mu(\cdot|A_{-k-1}^{-1})$. We call this coupling $\tP^*$ and drop the super-index $k$ from the random variables. By construction of the coupling,
\begin{align}
\tP^*(\cdot|H=h)=\hP_{A_{-k}^{-1-h},A_{-k-1}^{-1-h}}\left(\cdot\;\middle|\; \eta^1\in A_{-h}^{-1}, \xi_{-i}(\gamma_{-i})=0, i=1,...,h \right)
\end{align}
where $\hP_{A_{-k}^{-1-h},A_{-k-1}^{-1-h}}$ is the strong disagreement percolation coupling starting from the cone $C_{-1-h}$. In particular, $\eta^{1}$ and $\xi$ are independent. Denote by $G:=\{\xi_0(x)=0\ \forall\;x\in \bZ^d\}$ the good event that the disagreement process has become extinct by time 0.
We have
\begin{align}
\frac{\bP^{-(k+1)}_\mu\left(B \; \middle|\; A_{-k-1}^{-1} \right)}{\bP^{-k}_\mu\left(B \; \middle|\; A_{-k}^{-1} \right)}
& \geq \frac{\tP^{*}\left(\eta^2_0\in B , G \right)}{\tP^{*}\left(\eta^1_0\in B\right)}	
= \frac{\tP^{*}\left(\eta^1_0\in B , G \right)}{\tP^{*}\left(\eta^1_0\in B\right)}
= \tP^*(G).
\end{align}
Reversing the roles of $\eta^1$ and $\eta^2$, we also have
\begin{align}
\frac{\bP^{-(k+1)}_\mu\left(B \; \middle|\; A_{-k-1}^{-1} \right)}{\bP^{-k}_\mu\left(B \; \middle|\; A_{-k}^{-1} \right)} \leq \tP^*(G)^{-1}.
\end{align}
Since $\tP(G^c | H=h)$ is exponentially small in $h$, we have completed the proof once we show that $\tP^*(H\leq h)$ is exponentially small in $k$ for a fixed $h$. To see this, we look at $H^m$ in more detail. Since $H^m$ either increases by one or is reset to 0, $(H^m)$ is a time-inhomogeneous house-of-cards process with transition probability $\bP(H^{m+1}=h+1|H^m=h)=q_m(h)$. We have that $q_m(h)$ equals
\begin{align}
 &\frac{\tP^{m,*}\left(\xi^m_{-k+m}(\gamma_{-k+m})=0, \eta^{1,m}\in A_{-k+m}^{-k+m} \;\middle|\;H^m=h\right)}{\max\left(\tP^{m,*}(\eta^{1,m}\in A_{-k+m}^{-k+m} \;\middle|\;H^m=h),\tP^{m,*}(\eta^{2,m}\in A_{-k+m}^{-k+m} \;\middle|\;H^m=h)\right)}\\
= &\frac{\tP^{m,*}\left(\xi^m_{-k+m}(\gamma_{-k+m})=0 \;\middle|\;H^m=h\right)}{\max\left(1, \tP^{m,*}(\eta^{2,m}\in A_{-k+m}^{-k+m} \;\middle|\;H^m=h)(\tP^{m,*}(\eta^{1,m}\in A_{-k+m}^{-k+m} \;\middle|\;H^m=h) )^{-1}\right)}\\
= &\frac{\tP^{m,*}\left(\xi^m_{-k+m}(\gamma_{-k+m})=0 \;\middle|\;H^m=h\right)}{\max\left(1, 
\tP^{m,*}(\xi^m_{-k+m}(\gamma_{-k+m})=0 \;\middle|\;H^m=h) + \tP^{m,*}(\eta^{2,m} /\eta^{1,m}) \right)}
\\ \geq & \frac{\tP^{m,*}\left(\xi^m_{-k+m}(\gamma_{-k+m})=0 \;\middle|\;H^m=h\right)}{1+ (\inf_{i,\eta}c_i(\eta))^{-1}\left(1-
\tP^{m,*}\left(\xi_{-k+m}(\gamma_{-k+m})=0\;\middle|\;H^m=h\right)\right)},
\end{align}
where in the last equality before the inequality we denoted by
\begin{align}
\tP^{m,*}\left(\eta^{2,m} /\eta^{1,m}\right) :=
\frac{\tP^{m,*}\left(\eta^{2,m}\in A_{-k+m}^{-k+m}, \xi^m_{-k+m}(\gamma_{-k+m})=1\;\middle|\;H^m=h\right)}{\tP^{m,*}\left(\eta^{1,m}\in A_{-k+m}^{-k+m} \;\middle|\;H^m=h\right)}.
\end{align}
Note that $\inf_{i,\eta}c_i(\eta)>0$, since $\inf_{i,\eta}c_i(\eta)=0$ implies $p^*=1>p_c$. Conditioned on $H^m=h$ the probability of $\xi^m_{-k+m}(\gamma_{-k+m})=0$ is larger than the probability that there is no percolation path from $C_{-k+m-h-1}$ to $(\gamma_{-k+m},-k+m)$, which converges exponentially fast to 1 in $h$.
Therefore there are constants $c_1,c_2>0$ so that $1-q_m(h)\leq c_1e^{-c_2 h}$. We also have $q_m(h)\geq (1-p^*)\inf_{i,\eta}c_i(\eta)$. Those two facts imply that $\tP^*(H\leq h)\leq c_3e^{-c_4(k-c_5 h)}$ for some constants $c_3,c_4,c_5>0$.  
\end{proof}

Corollary \ref{cor SDP} follows as a direct consequence of Theorems \ref{thm alternative EP} and  \ref{thm:sdp}.
\begin{proof}[Proof of Corollary \ref{cor SDP}]
Let $l \in \nat$ and consider $A_l \in \mathcal{A}_{-l}^{-1}$. By telescoping, for any $B \in \mathcal{F}_{=0}$, we have by Theorem \ref{thm:sdp} that 
\begin{align}
\left| \frac{\bP_\mu (B \mid A_l)}{\mu(B)}^{\pm} -1 \right| \leq \left[\prod_{i=1}^k (1+C\delta^i)\right]. 
\end{align}
Setting $M:=\prod_{i=1}^\infty (1+C\delta^i)$, and noting that the statement (and proof) of Theorem \ref{thm alternative EP} holds when $\mathcal{F}_{\geq0}$ is replaced by $\mathcal{F}_{=0}$, we  conclude the proof.
\end{proof}



\subsection*{Acknowledgement}
The authors thanks Markus Heydenreich and Noam Berger for useful discussions and comments, and Matthias Birkner for careful proof reading. S.A. Bethuelsen thanks LMU Munich for hospitality during the writing of the paper.
S.A. Bethuelsen was supported by the Netherlands Organization for Scientific Research (NWO).



\end{document}